\documentclass[a4paper,twoside]{article}  

\usepackage{amssymb,amsmath,amsthm,dsfont,amsfonts,color,latexsym}

\usepackage{hyperref}

\usepackage{mathrsfs} 

\definecolor{blue}{rgb}{0.00,0.00,1.00}
\definecolor{red}{rgb}{1.00,0.00,0.00}

\allowdisplaybreaks[4]
\renewcommand{\baselinestretch}{1.2}
\def\bq{\begin{equation}}
\def\eq{\end{equation}}
\def\ba{\begin{array}{ccc}}
\def\bal{\begin{array}{lll}}
\def\ea{\end{array}}
\def\bsp{\begin{split}}
\def\esp{\end{split}}

\def\({\left(}\def\){\right)}
\def\[{\left[}\def\]{\right]}

    \def \R   {\mathbb{R}}
    
    \def\C    {\mathbb{C}}

    \def\eps  {\epsilon}
    \def\intr {\int_{\R^3}}
    
    \def\intt {\int^t_0}

    \def \Q    {\mathcal{Q}}
    
    \def \N    {\mathbb{N}}

    \def \pt   {\partial}
    
    \def \Dt   {\frac{\rm d}{{\rm d}t}}
    
    \def \dt    {\partial_t}

    \def \dx    {\partial_x}

    \def \dxa   {\partial^{\alpha}_x}

    \def\Tdx   {\nabla_x}
    \def\Tdv   {\nabla_v}

       \def\bq{\begin{equation}}
       \def\eq{\end{equation}}
       \def\be{\begin{equation}}
       \def\ee{\end{equation}}
       \def\bma#1\ema{{\allowdisplaybreaks\begin{align}#1\end{align}}}
       \def\bmas#1\emas{{\allowdisplaybreaks\begin{align*}#1\end{align*}}}
       \def\bln#1\eln{{\allowdisplaybreaks\begin{aligned}#1\end{aligned}}}
       \def\nnm{\notag}
       \def\bgr#1\egr{\allowdisplaybreaks\begin{gather}#1\end{gather}}
       \def\bgrs#1\egrs{\allowdisplaybreaks\begin{gather*}#1\end{gather*}}

       \theoremstyle{plain}
       \newtheorem{lem}{\bf Lemma}[section]
       \newtheorem{thm}[lem]{\textbf{Theorem}}

\begin{document}


\title{Green's Function and Pointwise  Behaviors of the  One-Dimensional modified Vlasov-Poisson-Boltzmann System}

\author{ Yanchao Li$^1$,\, Mingying Zhong$^2$,\, \\[2mm]
 \emph
    {\small\it  $^1$School of Physical Science and Technology,
    Guangxi University, P.R.China.}\\
    {\small\it E-mail:\ liyanchaozjl@$163$.com}\\
    {\small\it  $^2$School of  Mathematics and Information Sciences,
    Guangxi University, P.R.China.}\\
    {\small\it E-mail:\ zhongmingying@sina.com}\\[5mm]
    }
\date{ }

\pagestyle{myheadings}
\markboth{modified Vlasov-Poisson-Boltzmann System}%
{Y.-C. Li, M.-Y. Zhong}

 \maketitle

 \thispagestyle{empty}

\begin{abstract}\noindent{The pointwise space-time behaviors of the Green's function and the global solution to the modified Vlasov-Poisson-Boltzmann (mVPB) system in one-dimensional space are studied in this paper. It is shown that, the Green's function admits the diffusion wave, the Huygens's type sound wave, the singular kinetic wave and the remainder term decaying exponentially in space-time. These behaviors are similar to the Boltzmann equation (Liu and Yu in Comm. Pure Appl. Math. 57: 1543-1608, 2004). Furthermore, we establish the pointwise space-time nonlinear diffusive behaviors of the global solution to the nonlinear mVPB system in terms of the Green's function.
}

\medskip



\end{abstract}

\tableofcontents


\section{Introduction}
\label{sect1}
\setcounter{equation}{0}

The  Vlasov-Poisson-Boltzmann  system is used to model the time
evolution of dilute charged particles  (e.g., electrons and ions) in the absence of an external magnetic field \cite{MARKOWICH-1}. Assume that the electron density is very rarefied and reaches a local equilibrium state with small electron mass compared with the ions, the motion of ions can be described by the  following modified Vlasov-Poisson-Boltzmann (mVPB) system \cite{Codier-Grenier}:
\bma
&\dt F+v_{1}\dx F+\dx \Phi\partial_{v_{1}} F =\Q(F,F),\label{mVPB1}\\
&\partial_{xx}\Phi=\intr Fdv-e^{-\Phi},\label{mVPB2}
\ema
where $F=F(t,x,v)$ is the distribution function of ions with $(t,x,v)\in\R_+\times\R^1_x\times\R^3_v$, and $\Phi(t,x)$ denotes the
electric potential. The collision between particles is given by the standard Boltzmann collision operator $\Q(f,g)$ as below
\bq
\mathcal{Q}(f,g)=\frac{1}{2}\int_{\mathbb{R}^3}\int_{\mathbb{S}^2}|(v-v_{\ast})\cdot\omega|(f'_{\ast}g'+f'g'_{\ast}-f_{\ast}g-fg_{\ast})dv_{\ast}d\omega,
\eq
where
\bmas
&f'_{\ast}=f(t,x,v'_{\ast}),\quad f'=f(t,x,v'),\quad f_{\ast}=f(t,x,v_{\ast}),\quad f=f(t,x,v), \\
&v'=v-[(v-v_{\ast})\cdot\omega]\omega,\quad v'_{\ast}=v_{\ast}-[(v-v_{\ast})\cdot\omega]\omega, \quad\omega\in\mathbb{S}^2.
\emas

There has been much important progress made recently on the well-posedness and asymptotical behaviors of solution to the Vlasov-Poisson-Boltzmann(VPB) system. In particular, the global existence of a renormalized weak solution for general large initial data was shown in \cite{Lions-1,Lions-2,MISCHLER-1}. The global existence of a unique strong solution with the initial data near the normalized global Maxwellian was obtained in spatial period domain \cite{GUO-3} and in spatial whole space \cite{DUAN-1,YANG-1,YANG-2} for hard sphere, and then in \cite{DUAN-3,DUAN-4} for hard potential or soft potential. The existence of a classical solution near a vacuum was investigated in \cite{GUO-2}. 

The long time behaviors of the global solutions to the VPB systems were studied in \cite{DUAN-2,LI-2,LI-1,YANG-3}. Indeed, the decay rate $(1+t)^{-\frac{1}{4}}$ in $L^2$ norm for one-species VPB system in $\R^3$ was obtained in \cite{DUAN-2,LI-1}, and the decay rate $(1+t)^{-\frac{3}{4}}$ in $L^2$ norm for two-species VPB and modified VPB systems in $\R^3$ was obtained in \cite{LI-2,YANG-3}. The spectrum analysis and the optimal time-decay rate of the global solutions to the VPB systems for both one-species and two-species were studied in \cite{LI-2,LI-1}. The pointwise space-time behaviors of the Green's function and the global solution to the one-species VPB system was studied in \cite{LI-5}. Moreover, the wave phenomena is observed for one-dimensional VPB system  \cite{DUAN-5,LI-3,LI-4}, such as the shock profile, rarefaction wave and viscous contact wave.

In this paper, we study the pointwise space-time behaviors of the Green's function and the global solution to the one-dimensional  mVPB system \eqref{mVPB1}--\eqref{mVPB2} based on the spectral analysis \cite{LI-2}. To begin with, let us consider the Cauchy problem for the one-dimensional mVPB
system \eqref{mVPB1}--\eqref{mVPB2} with the following initial data:
\bq
F(0,x,v)=F_0(x,v),\quad (x,v)\in\mathbb{R}^1_x\times\mathbb{R}^3_v.\label{mVPB3}
\eq

The mVPB system \eqref{mVPB1}--\eqref{mVPB2} has an equilibrium state $(F_*,\Phi_*)=(M(v),0)$ with the normalized Maxwellian $M(v)$ defined by
$$
M(v)=\frac{1}{(2\pi)^{\frac{3}{2}}}e^{-\frac{|v|^2}{2}},\quad v\in\R^3.
$$

Define the perturbation  $f(t,x,v)$ of $F(t,x,v)$ near  $M$ by
$$
F=M+\sqrt M f,
$$
then the mVPB system \eqref{mVPB1}, \eqref{mVPB2} and \eqref{mVPB3}  is reformulated into
\bma
 &\dt f+v_{1}\dx  f-v_{1}\sqrt{M}\dx  \Phi-Lf
=\frac12 (v_{1}\dx \Phi)f-\dx  \Phi\partial_{v_{1}} f+\Gamma(f,f),\label{mVPB8}\\
&(I-\partial_{xx})\Phi=-\intr f\sqrt{M}dv+(e^{-\Phi}+\Phi-1),\label{mVPB9}\\
&f(0,x,v)=f_0(x,v)=(F_0-M){M}^{-1/2},\label{mVPB10}
 \ema
 where the linearized collision operator $Lf$ and the nonlinear term $\Gamma(f,f)$ are defined by
 \bma
&Lf=\frac1{\sqrt M}[\Q(M,\sqrt{M}f)+\Q(\sqrt{M}f,M)],  \label{Lf}\\
&\Gamma(f,f)=\frac1{\sqrt M}\Q(\sqrt{M}f,\sqrt{M}f).   \label{gf}
 \ema

We have, cf \cite{CERCIGNANI-1,Glassey-1},
 \be \left\{\bln\label{LKF}
 &(Lf)(v)=(Kf)(v)-\nu(v) f(v),   \quad
 (Kf)(v)=\intr k(v,v_*)f(v_*)dv_*, \\
&k(v,v_*)=\frac2{\sqrt{2\pi}|v-v_*|}e^{-\frac{(|v|^2-|v_*|^2)^2}{8|v-v_*|^2}-\frac{|v-v_*|^2}8}-\frac{|v-v_*|}{2\sqrt{2\pi}}e^{-\frac{|v|^2+|v_*|^2}4},\\
&\nu(v)= \sqrt{2\pi} \bigg(e^{-\frac{|v|^2}2}+ \(|v|+\frac1{|v|}\)\int^{|v|}_0e^{-\frac{|u|^2}2}du\bigg),
\eln\right.
 \ee
where  $\nu(v)$,  the collision frequency, is a real  function, and
$K$ is a self-adjoint compact operator on $L^2(\R^3_v)$ with  real symmetric integral kernels $k(v,v_*)$. For hard sphere model, $\nu(v)$ satisfies
 \be
\nu_0(1+|v|)\leq\nu(v)\leq \nu_1(1+|v|),  \label{nuv}
 \ee
 with $\nu_1\geq \nu_0>0$ two constants.

The null space of the operator $L$, denoted by $N_0$, is a five-dimensional subspace
spanned by the orthogonal basis $\{ \chi_j,\ j=0,1,2,3,4\}$  with
 \bq
  \chi_0=\sqrt{M},\quad  \chi_j=v_j\sqrt{M},\ j=1,2,3,\quad \chi_4=\frac{(|v|^2-3)\sqrt{M}}{\sqrt{6}}.\label{basis}
 \eq

Let   $L^2(\R^3)$ be a Hilbert space of complex-value functions $f(v)$
 with the inner product and the norm
$$
(f,g)=\intr f(v)\overline{g(v)}dv,\quad \|f\|=\(\intr |f(v)|^2dv\)^{1/2}.
$$

Introduce the macro-micro decomposition as follows
\be \label{P10}
\left\{\bal
f=P_0f+P_1f,\\
P_0f=P^1_0f+P^2_0f+P^3_0f,\,\,\,  P_1f=f-P_0f,\\
P^1_0f=(f, \chi_0) \chi_0,\,\,\, P^3_0f=(f,\chi_4)\chi_4,\\
P^2_0f=\sum^3_{k=1}(f,\chi_k)\chi_k.
\ea\right.
\ee

From the Boltzmann's H-theorem, $L$ is non-positive and
moreover, $L$ is locally coercive in the sense that there is a constant $\mu>0$ such that
\be
 (Lf,f) \leq -\mu \| P_1f\|^2, \quad   f\in D(L),\label{L_3}
 \ee
where $D(L)$ is the domains of $L$ given by
$$ D(L)=\left\{f\in L^2(\R^3)\,|\,\nu(v)f\in L^2(\R^3)\right\}.$$

Without the loss of generality, we assume in this paper that $\nu(0)\geq \nu_0\geq \mu>0$.

Since we only consider the pointwise behavior with respect to the space-time variable $(t,x)$, it's convenient to regard the Green's function $G(t,x)$ as an  operator on $L^2(\R^3_v)$:
\be   \label{LmVPB}
 \left\{\bln
 &\dt G =B G , \\
 &G(0,x)=\delta(x)I_v,
 \eln\right.
 \ee
 where $I_v$ is the identity map on $L^2(\R^3_v)$ and the operator $B$ is defined by
\be
 B f=Lf-v_{1}\dx f -v_{1}\dx (I-\partial_{xx})^{-1} P^1_0f. \label{B0(x)}
\ee

Then, the solution for the initial value problem of the linearized mVPB  equation
 \bq
 \left\{\bln
 &\dt f=B f, \\
 &f(0,x,v)=f_0(x,v),   \label{mVPB}
 \eln\right.
 \eq
can be represented by
\be f(t,x)=G(t)\ast f_0=\int_{-\infty}^{+\infty} G(x-y,t)f_0(y)dy, \ee
where $f_0(y)=f_0(y,v).$

For any $(t,x)$ and $f\in L^2(\R^3_v)$, we define the $L^2$ norm of $G(t,x)$ by
\be \|G(t,x)\|=\sup_{\|f\|=1}\|G(t,x)f\|, \label{norm}\ee
and define the $L^2$ norm of a operator $T$ in $L^2(\R^3_v)$ as
\be \|T\|=\sup_{\|f\|=1}\|Tf\|. \label{norm1}\ee

\noindent\textbf{Notations}.\ Before stating the main results in this paper, we list some notations. The Fourier transform of $f=f(x,v)$
is denoted by
$$\hat{f}(\eta,v)=\mathcal{F}f(x,v)=\frac1{\sqrt{2\pi}}\int^{+\infty}_{-\infty} e^{-i x \eta}f(x,v)dx,$$
where and throughout this paper we denote $i=\sqrt{-1}$.

Set a weight function $w(v)$ by
$$
w(v)=(1+|v|^2)^{1/2},
$$
and the Sobolev space $ H^N_{k}=\{\,f\in L^2(\R^1_x\times \R^3_v)\,|\,\|f\|_{H^N_{k}}<\infty\}$
equipped with the norm
$$
 \|f\|_{H^N_{k}}=\sum_{|\alpha|\leq N}\|w^k\dxa f\|_{L^2(\R_x\times \R^3_v)}.
$$

In what follows,  denote by $\|\cdot\|_{L^2_{x,v}}$, $\|\cdot\|_{L^2_{\eta,v}}$ and $\|\cdot\|_{L^{\infty}_x}$ the norms of the function spaces $L^2(\R_x\times \R^3_v)$, $L^2(\R_{\eta}\times \R^3_v)$ and $L^{\infty}(\R_x)$
respectively, and by $\|\cdot\|_{L^2_x}$, $\|\cdot\|_{L^2_{\eta}}$ and $\|\cdot\|_{L^2_v}$  the norms of the function spaces $L^2(\R_x)$, $L^2(\R_{\eta})$ and $L^2(\R^3_v)$ respectively.

\medskip

First, we have the pointwise space-time behaviors of the Green's function for the linearized mVPB system \eqref{LmVPB} as follows:
\begin{thm}\label{mvpbth1}
Let $G(t,x)$ be the Green's function for the mVPB system defined by \eqref{LmVPB}. Then, the Green's function $G(t,x)$ can be decomposed into
$$
G(t,x)=G_1(t,x)+G_2(t,x)+W_{2}(t,x),
$$
where $W_2(t,x)$ is the singular kinetic wave, $G_1(t,x)$ is the fluid part at low frequency and $G_2(t,x)$ is the remainder part. For $\alpha\geq0$, there exist two positive constants $C$ and $D$ such that the fluid part $G_1(t,x)$ is smooth and satisfies
\bq\label{Thm1}
\left\{\bln
&\|\dxa P^j_0G_1(t,x)\|\leq C(1+t)^{-\frac{1+\alpha}{2}}\sum^1_{i=-1}e^{-\frac{|x-\beta_it|^2}{D(1+t)}}+Ce^{-\frac{|x|+t}{D}},\quad j=1,2,3,\\
&\|\dxa P_{1}G_1(t,x)\|,\|\dxa G_1(t,x)P_1\|\leq
C(1+t)^{-\frac{2+\alpha}{2}}\sum^{1}_{i=-1}e^{-\frac{|x-\beta_it|^2}{D(1+t)}}+Ce^{-\frac{|x|+t}{D}},\\
&\|\dxa P_{1}G_1(t,x)P_1\|\leq
C(1+t)^{-\frac{3+\alpha}{2}}\sum^{1}_{i=-1}e^{-\frac{|x-\beta_it|^2}{D(1+t)}}+Ce^{-\frac{|x|+t}{D}},
\eln\right.
\eq
where $\beta_j$, $j=-1,0,1$ is sound speed defined by
$$
\beta_{\pm 1}=\pm\sqrt{\frac{8}{3}},\quad  \beta_{0}=0.
$$

And the remainder part $G_2(t,x)$ is bounded and satisfies
\bq
\|G_2(t,x)\|\leq Ce^{-\frac{|x|+t}{D}}.\label{Thm2}
\eq

The singular kinetic wave $W_{2}(t,x)$ is defined by
\bq\label{Wk}
W_{2}(t,x)=\sum^{6}_{k=0}J_k(t,x),
\eq
with
$$
\left\{\bln
&J_0(t,x)=S^t\delta(x)I_v=e^{-\nu(v)t}\delta(x-v_1t)I_v,\\
&J_k(t,x)=\int^t_0S^{t-s}(K+v_1\dx (I-\partial_{xx})^{-1}P^1_0)J_{k-1}ds,\quad k\geq1.
\eln\right.
$$

Here $I_v$ is the identity map in $L^2(\mathbb{R}^3_v)$ and the operator $S^t$ is defined by
\bq\label{St}
S^tg(x,v)=e^{-\nu(v)t}g(x-v_1t,v).
\eq
\end{thm}

Then we have the pointwise space-time behavior of the global solution to the nonlinear mVPB system \eqref{mVPB8}--\eqref{mVPB10} as follows:

\begin{thm}\label{mvpbth2}
There exists a small constant $\delta_0>0$ such that if the initial data $f_0$ satisfies $\|f_0\|_{H^4_3}\leq\delta_0$ and
\bq\label{F0}
\|\dxa f_0(x)\|_{L^{\infty}_{v,3}}+\|\nabla_vf_0(x)\|_{L^{\infty}_{v,2}}\leq C\delta_0(1+|x|^2)^{-\gamma_0},
\eq
for $\alpha=0,1$ and $\gamma_0>1/2$, then there exists a unique global solution $(f,\Phi)$ to the mVPB system \eqref{mVPB8}--\eqref{mVPB10} satisfying
\bma
\|f(t,x)\|_{L^{\infty}_{v,3}}+\|\partial_{v_1}f(t,x)\|_{L^{\infty}_{v,2}}+|\Phi(t,x)|&\leq C\delta_0(1+t)^{-\frac{1}{2}}\sum^1_{i=-1}B_{\frac{1}{2}}(t,x-\beta_it),\label{NL5}\\
\|\dx f(t,x)\|_{L^{\infty}_{v,3}}+|\dx \Phi(t,x)|+|\partial_t\Phi(t,x)|&\leq C\delta_0(1+t)^{-1}\sum^1_{i=-1}B_{\frac{1}{2}}(t,x-\beta_it),\label{NL4}
\ema
where the space-time diffusive profile $B_k(t,x-\lambda t)$ is defined for any $k>0$ and $\lambda\in \R$ by
$$
B_k(t,x-\lambda t)=\(1+\frac{|x-\lambda t|^2}{1+t}\)^{-k},\quad (t,x)\in\mathbb{R}_{+}\times\mathbb{R}.
$$
\end{thm}

The results in Theorem \ref{mvpbth1} on the pointwise behavior of the Green's function $G$ to the mVPB system \eqref{mVPB8}--\eqref{mVPB9} is proved based on the spectral analysis \cite{LI-2} and the ideas inspired by \cite{LI-5,LIU-2,LIU-3}. First, we estimate the Green's function $G$ inside the Mach region $|x|\leq6t$ based on the spectral analysis. Indeed, we decompose the Green's function $G $ into the lower frequency part $G_L$ and the high frequency part $G_H$, and further split $G_L$ into the fluid parts $G_{L,0}$ and the non-fluid parts $G_{L,1}$ respectively, namely
$$
\left\{\bln
&G=G_L+G_H,\\
&G_L=G_{L,0}+G_{L,1}.
\eln\right.
$$

By using Fourier analysis techniques, we can show that the low-frequency fluid part $G_{L,0}(t,x)$ is smooth and contains Huygens waves and diffuse waves for $|x|\leq6t$ since the Fourier transform of the linear mVPB operator $\hat{B}(\eta)$ has five eigenvalues $\{\lambda_j(\eta),\ j=-1,0,1,2,3\}$ at the low frequency region $|\eta|\leq r_0$ satisfying
\be
\lambda_j(\eta)=-i\beta_j\eta-a_j\eta^2+O(\eta^3),\quad a_j>0.
\ee

Then, we apply a Picard's iteration to estimate the Green's function $G$ outside the Mach region $|x|>6t$.
Since $\hat{G}(t,\eta)$ does not belong to $L^1_{\eta}(\R)$,  $G(t,x)$ can be decomposed into the singular waves and the bounded remainder part. To exact the singular part from $G(t,x)$, we defined the approximate sequence $\{\hat{J}_{k},\ k\ge 0\}$ for the Green's function $\hat{G}(t,\eta)$, where $\hat{J}_{k}$ can be represented by the combination of the mixture operator $\hat{\mathbb{M}}_{k}^t(\eta)$. From the Mixture lemma (refer to \cite{LI-5,LIU-2}), $\hat{\mathbb{M}}^t_k(\eta)$ is analytic in $ \{\eta\in\C\,|\, |{\rm Im}\lambda|\le \delta_0\}$ and satisfies
$$
\|\hat{\mathbb{M}}^t_{3k}(\eta)\|\leq C_k(1+t)^{3k}(1+|\eta|)^{-k}e^{- \nu_0t}.
$$

This together with Cauchy theorem implies that  the approximate solution $J_{6}(t,x)$ is bounded and satisfies
\be
\|J_{6}(t,x)\|\leq Ce^{-\frac{\nu_{0}(|x|+t)}{4}}. \label{J6}
\ee

We define the singular part $W_{k}(t,x)$ as
$$
W_{k}(t,x)=\sum^{3k}_{i=0}J_i(t,x).
$$

Thus, \eqref{J6} implies that the remainder part $R_2(t,x)=G(t,x)-W_2(t,x)$ is bounded for all $(t,x)$. By choosing the weighted function as $w=e^{\epsilon(|x|-Yt)}$ and using the fact that for any $f=f(x)$ and $\delta\in (0,2)$ (refer to Lemma \ref{mvpbgf6}),
$$\int_{\mathbb{R}}|\partial^\alpha_x(I-\partial_{xx})^{-1}f|^2e^{\delta|x|}dx\leq \frac4{(2-\delta)^2}\int_{\mathbb{R}}|f|^2e^{\delta|x|}dx, $$
we can show by the weighted energy method that the remainder part $R_{2}(t,x)=G(t,x)-W_2(t,x)$ satisfies (refer to Lemma \ref{mvpbgf7})
$$
\|G(t,x)-W_{2}(t,x)\|\leq Ce^{-\frac{|x|+t}{D}}.
$$

Applying the above decompositions and estimates, we can obtain the decomposition and pointwise space-time behavior for each part of the Green's function $G(t,x)$ as listed Theorem \ref{mvpbth1}.

Finally, by using the estimates of the Green's functions in Theorem \ref{mvpbth1}, the energy estimate in Lemmas \ref{mvpbpw5}--\ref{mvpbpw6} and the estimates of waves coupling in Lemmas \ref{mvpbpw8}--\ref{mvpbpw12}, one can establish the pointwise space-time estimate on the global solution to the nonlinear mVPB system given in Theorem \ref{mvpbth2}.

The rest of this paper is organized as follows. In section \ref{sect2}, we present results about the spectrum analysis of the linear operator related to the linearized mVPB system. In section \ref{sect3}, we establish the pointwise space-time estimates of the Green's function to the linearized mVPB system. In section \ref{sect4}, we prove the pointwise space-time estimates of the global solution to the original nonlinear mVPB system.


\section{Spectral analysis}
\label{sect2}
\setcounter{equation}{0}

In this section, we show the spectrum structure of
the linearized  mVPB  operator~\eqref{B(s)} and analyze the analyticity and expansion of the eigenvalues and eigenfunctions in order to study the pointwise estimate of the Green's function.

First, we take the Fourier transform in \eqref{LmVPB} with respect to $x$ to have
 \bq
 \left\{\bln
 &\dt \hat{G}=\hat{B}(\eta)\hat{G},  \quad t>0, \\
 &\hat{G}(0,\eta)=1(\eta)I_v,  \quad (\eta,v)\in \R^1\times \R^3, \label{mVPB7}
 \eln\right.
 \eq
where the operator $\hat{B}(\eta)$ is defined by
 \bq
\hat{B}(\eta) =L-i v_{1}\eta -\frac{i v_{1}\eta }{1+\eta^2}P^1_{0}.  \label{B(s)}
 \eq

Introduce the weighted Hilbert space $L^2_{\eta}(\R^3_v)$ as
$$
 L^2_{\eta}(\R^3)=\{f\in L^2(\R^3_v)\,|\,\|f\|_{\eta}=\sqrt{(f,f)_{\eta}}<\infty\},
$$
equipped with the inner product
$$
 (f,g)_{\eta}=(f,g)+\frac1{1+\eta^2}(P^1_0 f,P^1_0 g).
$$

 Let $\sigma(\hat{B}(\eta))$ and $\rho(\hat{B}(\eta))$ denote the spectrum set and the resolvent set of the operator $\hat{B}(\eta)$ respectively. First, we have the following  results about the spectrum set and the resolvent set of  $\hat{B}(\eta)$.

\begin{lem}[\cite{LI-2}]\label{mvpbsp0}
The following results hold.\\
(1) For any $\delta>0$ and all $\eta\in\R$, there exists $y_1(\delta)>0$ such that
\be
\rho(\hat{B}(\eta))\supset\{\lambda\in\mathbb{C}\,|\,\mathrm{Re}\lambda\geq-\nu_0+\delta,\,|\mathrm{Im}\lambda|\geq y_1\}\cup\{\lambda\in\mathbb{C}\,|\,\mathrm{Re}\lambda>0\}.
\ee
(2) For any $r_0>0$, there exists $\alpha=\alpha(r_0)>0$ such that for $|\eta|\geq r_0$,
\be
\sigma(\hat{B}(\eta))\subset\{\lambda\in\mathbb{C}\,|\,\mathrm{Re}\lambda<-\alpha\}.
\ee
(3) For any $\delta>0$, there exists $r_1(\delta)>0$ such that for $ |\xi|\leq r_1$,
  \bq
\sigma(\hat{B}(\eta))\cap\{\lambda\in\mathbb{C}\,|\,\mathrm{Re}\lambda\ge-\mu/2\}\subset
\{\lambda\in\mathbb{C}\,|\,|\lambda|<\delta\}.\label{m_eigen1}
\eq
\end{lem}

Then, we show the analyticity and expansion of the eigenvalues and  eigenfunctions to the operator $\hat{B}(\eta)$ at low frequency.


\begin{lem} \label{mvpbsp1}

(1) There exists a constant $r_0>0$ such that for $|\eta|\leq r_{0}$,
$$ \sigma(\hat{B}(\eta))\cap \{\lambda\in \C\,|\,{\rm Re}\lambda>-\mu/2\}=\{\lambda_j(\eta),\ j=-1,0,1,2,3\}.$$

In particular, the eigenvalues $\lambda_j(\eta)$ are analytic functions of $\eta$ and satisfy the following asymptotic expansion for $|\eta|\le r_0:$
\be                       \label{specr0}
\left\{\bln
&\lambda_{\pm1}(\eta)=\mp i\sqrt{\frac{8}{3}}\eta- a_{\pm1}\eta^{2}+O(\eta^{3}), \quad \overline{\lambda_1(\eta)}=\lambda_{-1}(\eta),\\
&\lambda_{0}(\eta)=-a_{0}\eta^{2}+O(\eta^{3}),\\
&\lambda_l(\eta)=-a_l\eta^2+O(\eta^3),\quad l=2,3,
\eln\right.
\ee
where $a_{j}>0,\ j=-1,0,1,2,3$ is defined by
\bq
\left\{\bln
&a_j=-(L^{-1}P_1v_1E_j,v_1E_j)>0, \\
&E_{\pm1}=\frac{\sqrt{3}}{4}\chi_0\pm\frac{\sqrt{2}}{2}\chi_1+\frac{\sqrt{2}}{4}\chi_4,\\
&E_0=\frac{\sqrt{2}}{4}\chi_{0}-\frac{\sqrt{3}}{2}\chi_{4},\\
&E_l=\chi_l,\quad l=2,3.
\eln\right.
\eq
(2) The eigenfunctions $\psi_j(\eta)=\psi_j(\eta,v)$ for $j=-1,0,1,2,3$ are analytic in $\{\eta\in\mathbb{C}\,|\,|\eta|\leq r_0\}$ and satisfy

\be\label{specr1}
\left\{\bln
&(\psi_j(\eta),\overline{\psi_k(\eta)})_{\eta}=\delta_{jk},\quad  j,k=-1,0,1,2,3,\\
&\psi_j=\psi_{j,0}+\psi_{j,1}\eta+O(\eta^{2}),
\eln\right.
\eq
where the coefficients $\psi_{j,0},\,\psi_{j,1}$ are given as
\bq\label{TH1}
\left\{\bln
&\psi_{j,0}=E_j, \quad j=-1,0,1,2,3,\\
&\psi_{l,1}=\sum^1_{k=-1}b^l_kE_k+iL^{-1}P_1v_1E_l,\quad l=-1,0,1,\\
&\psi_{k,1}=iL^{-1}P_1v_1E_k,\quad k=2,3,
\eln\right.
\eq
and $b^j_k,~j,k=-1,0,1$ are defined by
\bq
\left\{\bln
&b^j_j=0,\quad  b^j_k=\frac{i(L^{-1}P_1v_1E_j,v_1E_k)}{(\beta_j-\beta_k)},\quad j\neq k,\\
&\beta_{\pm1}=\mp\sqrt{\frac{8}{3}},\quad  \beta_{0}=0.
\eln\right.
\eq
\end{lem}
\begin{proof}

We consider the eigenvalue problem in the form
$$
\hat{B}(\eta)\psi=\lambda \psi,
$$
that is
\be
\(L-iv_1\eta -\frac{iv_1\eta }{1+\eta^2}P^1_0\)\psi=\lambda\psi. \label{spj1}
\ee

By the macro-micro decomposition, the eigenfunction $\psi$ of \eqref{spj1} can be divided into
$$
\psi=P_0\eta+P_1\eta=g_0+g_1.
$$

Let $\lambda=-i\eta\sigma$. Hence, \eqref{spj1} gives
\bma
&-i\eta\sigma g_0=-P_0[(iv_1\eta )(g_0+g_1)]-\frac{iv_1\eta }{1+\eta^2}P^1_0g_{0},\label{e0}\\
&-i\eta\sigma g_1=Lg_1-P_1[iv_1\eta (g_0+g_1)].\label{e1}
\ema

According to \eqref{e1}, $g_1$ can be represented by
\bq
g_1=i\eta(L+i\eta\sigma -i\eta P_1v_1)^{-1}P_{1}v_1g_0.\label{e1b}
\eq

By substituting \eqref{e1b} into \eqref{e0}, we have
\bq
\sigma g_0=P_0v_1g_0+\frac{v_1}{1+\eta^2}P^1_0g_0+i\eta P_0[v_1(L+i\eta\sigma-i\eta P_1v_1)^{-1}P_1v_1g_{0}].\label{e0b}
\eq

Define the operator $A(\eta)=P_0v_1P_0+\frac{v_1}{1+\eta^2}P^1_0$. We have matrix representation of $A(\eta)$ as follows
\bq
A(\eta)=\left(
  \begin{array}{ccccc}
    0 & 1 & 0 & 0 & 0 \\
   1+\frac{1}{1+\eta^2} & 0 & 0 & 0 & \sqrt{\frac{2}{3}} \\
    0 & 0 & 0 & 0 & 0 \\
    0 & 0 & 0 & 0 & 0 \\
    0 & \sqrt{\frac{2}{3}} & 0 & 0 & 0 \\
  \end{array}
\right).
\eq

It can be verified that the eigenvalue $u_j(\eta)$ and  eigenvectors $E_j(\eta)$ of $A(\eta)$ are given by
\bq\label{EUJ}
\left\{\bln
&u_{\pm1}(\eta)=\mp\sqrt{\frac{5}{3}+\frac{1}{1+\eta^2}},\quad u_j(\eta)=0, \quad j=0,2,3,\\
&E_{\pm1}(\eta)=\frac{1}{\sqrt{\frac{10}{3}+\frac{2}{1+\eta^2}}}\chi_0\mp\frac{\sqrt{2}}{2}\chi_1+\frac{1}{\sqrt{5+\frac{3}{1+\eta^2}}}\chi_4,\\
&E_0(\eta)=\frac{\sqrt{\frac{2}{3}}}{\sqrt{\frac{2}{3}(1+\frac{1}{1+\eta^2})+(1+\frac{1}{1+\eta^2})^2}}\chi_0-\frac{\sqrt{1+\frac{1}{1+\eta^2}}}{\sqrt{\frac{5}{3}+\frac{1}{1+\eta^2}}}\chi_4,\\
&E_k(\eta)=\chi_k,\quad k=2,3,\\
&(E_j(\eta),E_k(\eta))_{\eta}=\delta_{jk},\quad -1\leq j,k\leq3.
\eln\right.
\eq

Moreover, we denote $E_j=E_j(0)$ $(j=-1,0,1,2,3)$.

To solve the eigenvalue problem \eqref{e0b}, we write $g_0\in N_0$ in terms of the basis $E_j$ as
$$
g_0=\sum^{4}_{j=0}W_{j}E_{j-1}(\eta) \quad \text{with}\quad W_{j}=(\eta_0,E_{j-1}(\eta))_{\eta},\quad j=0,1,2,3,4.
$$

Taking the inner product $(\cdot,\cdot)_{\eta}$ between \eqref{e0b} and $E_j(\eta)$~$(j=-1,0,1,2,3)$ respectively, we obtain the equations about
$\sigma$ and $(W_0,W_1,W_2,W_3,W_4)$ for $-\mathrm{Re}(i\eta\sigma)>-\mu$:
\bq\label{WRJ}
\sigma W_j =u_{j-1}(\eta)W_j +i\eta \sum^4_{k=0}W_k R_{kj}(\sigma,\eta), \quad j=0,1,2,3,4,
\eq
where
\bq\label{EJJ}
R_{kj}(\sigma,\eta)=((L+i\eta\sigma-i\eta P_1v_1)^{-1}P_1v_1E_{k-1}(\eta),v_1E_{j-1}(\eta)).
\eq

In particular, it holds that
\bq\label{RJJ}
\left\{\bln
&R_{kl}(\sigma,\eta)=R_{lk}(\sigma,\eta)=0,\,\,\,k=0,1,2,\,\, l=3,4,\\
&R_{34}(\sigma,\eta)=R_{43}(\sigma,\eta)=0,\\
&R_{33}(\sigma,\eta)=R_{44}(\sigma,\eta).
\eln\right.
\eq

By \eqref{EJJ} and \eqref{RJJ}, we can divide \eqref{WRJ} into two systems:
\bma
\sigma W_j&=u_{j-1}(\eta)W_j+i\eta\sum^{2}_{i=0}W_iR_{ij}(\sigma,\eta),\quad j=0,1,2,\label{XJ1}\\
\sigma W_k&=i\eta W_kR_{33}(\sigma,\eta),\quad k=3,4.
\ema

Denote
\bma
D_0(\sigma,\eta)&=\sigma-i\eta R_{33}(\sigma,\eta),\label{D0}\\
D_1(\sigma,\eta)&=
\det\left(
\begin{array}{ccc}
\sigma-u_{-1}-i\eta R_{00}&-i\eta R_{10}&-i\eta R_{20}\\
-i\eta R_{01}&\sigma-u_0-i\eta R_{11}&-i\eta R_{21}\\
-i\eta R_{02}&-i\eta R_{12}&\sigma-u_1-i\eta R_{22}\\
\end{array}
\right).\label{D1}
\ema

The eigenvalues $\lambda=-i\eta\sigma$ can be solved by $D_0(\sigma,\eta)=0$ and $D_1(\sigma,\eta)=0$. By a direct computation and the implicit function theorem \cite{homo}, we can show
\begin{lem}\label{mvpbsp2}
 There exist two  constants $r_0,r_1>0$  so that the equation $D_0(\sigma,\eta)=0$ has a unique analytic solution $\sigma=\sigma(\eta)$ for $(\eta,\sigma)\in[-r_0,r_0]\times B_{r_1}(0)$ that satisfies
\bq
\sigma(0)=0,\quad \sigma'(0)=i(L^{-1}P_1v_1E_2,v_1E_2).
\eq
\end{lem}
We have the following result about the solution of $D_1(\sigma,\eta)=0$.

\begin{lem}\label{mvpbsp3}
There exist two constants $r_0,r_1>0$  so that the equation $D_1(\sigma,\eta)=0$ admits three analytic solutions $\sigma_j(\eta)\,(j=-1,0,1)$ for $(\eta,\sigma_j)\in[-r_0,r_0]\times B_{r_1}(u_j(0))$ that satisfy
\bq
\sigma_j(0)=u_j(0),\quad \sigma'_{j}(0)=i(L^{-1}P_1v_1E_j,v_1E_j).\label{si1}
\eq

Moreover, $\sigma_j(\eta)$ satisfies
\bq
-\sigma_j(-s)=\overline{\sigma_j(\eta)}=\sigma_{-j}(\eta),\quad j=-1,0,1.\label{si2}
\eq
\end{lem}
\begin{proof}
From \eqref{D1}, it holds that
\bq
\begin{split}
D_1(\sigma,0)&=
\det\left(
\begin{array}{ccc}
\sigma-u_{-1}(0)&0&0\\
0&\sigma-u_0(0)&0\\
0&0&\sigma-u_1(0)\\
\end{array}
\right)\\
&=(\sigma-u_{-1}(0))(\sigma-u_0(0))(\sigma-u_1(0)).
\end{split}
\eq

This implies that $D_1(\sigma,0)=0$ has three roots $\sigma_j=u_j(0)$ for $j=-1,0,1$. Since
\bmas
\partial_{\eta}D_1(\sigma,0)&=-iR_{00}(\sigma,0)(\sigma-u_0(0))(\sigma-u_1(0))\nnm\\
&\quad-iR_{11}(\sigma,0)(\sigma-u_{-1}(0))(\sigma-u_1(0))\nonumber\\
&\quad-iR_{22}(\sigma,0)(\sigma-u_{-1}(0))(\sigma-u_0(0)),\\
\partial_{\sigma}D_1(\sigma,0)&=(\sigma-u_0(0))(\sigma-u_1(0))+(\sigma-u_{-1}(0))(\sigma-u_1(0))\nnm\\
&\quad+(\sigma-u_0(0))(\sigma-u_1(0)),
\emas
it follows that
$$
\partial_{\sigma}D_1(u_j(0),0)\neq0.
$$

The implicit function theorem implies \cite{homo} that there exist two small constants $r_0,r_1>0$ and a unique analytic function $\sigma_j(\eta):[-r_0,r_0]\rightarrow B_{r_1}(u_j(0))$ so that $D_1(\sigma_j(\eta),\eta)=0$ for $\eta\in[-r_0,r_0]$, and in particular
\bq
\sigma_j(0)=u_j(0),\quad \sigma_j'(0)=-\frac{\partial_{\eta}D_1(u_j(0),0)}{\partial_{\sigma}D_1(u_j(0),0)}=iR_{jj}(u_j(0),0),\quad  j=-1,0,1.
\eq

This proves \eqref{si1}.

Since $R_{ij}(\beta,-\eta)=R_{ij}(-\beta,\eta)$, $R_{ij}(\overline{\beta},\eta)=\overline{R_{ij}(\beta,\eta)}$ for $i,j=0,1,2$, we obtain by \eqref{D1} that $D_1(-\sigma,\eta)=D_1(\sigma,-\eta)$ and $D_1(\overline{\sigma},\eta)=\overline{D_1(\sigma,\eta)}$. This together with the fact that $\sigma_j(\eta)=u_j(0)+O(\eta)$ for $j=-1,0,1$ and $|\eta|\leq r_0$, imply \eqref{si2}.
\end{proof}

The eigenvalues $\lambda_j(\eta)$ and the eigenfunctions $\psi_j(\eta)$ for $j=-1,0,1,2,3$ can be constructed as follows. For $j=2,3$, we take $\lambda_j=-i\eta\sigma(\eta)$ with $\sigma(\eta)$ being the solution to the equation $D_0(\sigma,\eta)=0$ defined in Lemma \ref{mvpbsp2}. Thus the corresponding eigenfunctions $\psi_j(\eta),\,j=2,3$ are defined by
\bq\label{EBJ}
\psi_j(\eta)=b_j(\eta)E_j(\eta)+i\eta b_j(\eta)[L-\lambda_j-i\eta P_1v_1]^{-1}P_1v_1E_j(\eta).
\eq

It is easy to verify that $\psi_2$ and $\psi_3$ are orthonormal, i.e., $(\psi_2(\eta),\overline{\psi_3(\eta)})_{\eta}=0$.

For $j=-1,0,1$, we take $\lambda_j=-i\eta\sigma_j(\eta)$ with $\sigma_j(\eta)$ being the solutions to the equation $D_1(\sigma,\eta)=0$ given in Lemma \ref{mvpbsp3}. Denote by $\{W_0^j,W_1^j,W_2^j\}$ a solution system \eqref{XJ1} for $\sigma=\sigma_j(\eta)$. Then we can construct the eigenfunctions $\psi_j(\eta),\, j=-1,0,1$ as
\bq\label{eta}
\left\{\bln
&\psi_j(\eta)=P_0\psi_j(\eta)+P_1\psi_j(\eta),\\
&P_0\psi_j(\eta)=W^j_0(\eta)E_{-1}(\eta)+W^j_1(\eta)E_0(\eta)+W^j_2(\eta)E_1(\eta),\\
&P_1\psi_j(\eta)=i\eta[L-\lambda_j-i\eta P_1v_1]^{-1}P_1v_1P_0\psi_j(\eta).
\eln\right.
\eq

We write
\bq\label{etaj1}
\(L-iv_1\eta -\frac{iv_1\eta }{1+\eta^2}P^1_0\)\psi_j(\eta)=\lambda_j(\eta)\psi_j(\eta),\quad -1\leq j\leq3.
\eq

Taking the inner product $(\cdot,\cdot)_{\eta}$ between  \eqref{etaj1} and $\overline{\psi_j(\eta)}$, and using the facts that
\bq
\begin{split}
&(\hat{B}(\eta)f,g)_{\eta}=(f,\hat{B}(-\eta)g)_{\eta},\quad f,g\in D(\hat{B}(\eta)),\\
&\hat{B}(-\eta)\overline{\psi_j(\eta)}=\overline{\lambda_j(\eta)}\cdot\overline{\psi_j(\eta)},
\end{split}
\eq
we have
$$
(\lambda_j(\eta)-\lambda_k(\eta))(\psi_j(\eta),\overline{\psi_k(\eta)})_{\eta}=0,\quad -1\leq j\neq k\leq3.
$$

For $\eta\neq0$ being sufficiently small, $\lambda_j(\eta)\neq\lambda_k(\eta)$ for $-1\leq j\neq k\leq2$. Therefore, we have
$$
(\psi_j(\eta),\overline{\psi_k(\eta)})_{\eta}=0,\quad -1\leq j\ne k\leq3.
$$

We can normalize them by taking $(\psi_j(\eta),\overline{\psi_j(\eta)})_{\eta}=1$ for $-1\leq j\leq3$.

The coefficients $b_j(\eta)$ for $j=2,3$ defined in \eqref{EBJ} are determined by the normalization condition as
\bq
b_j(\eta)^2(1+\eta^2D_j(\eta))=1,\label{EBJ1}
\eq
where
$$
D_j(\eta)=\((L-i\eta P_1v_1-\lambda_j(\eta))^{-1}P_1v_1E_j,(L+i\eta P_1v_1-\overline{\lambda_j(\eta)})^{-1}P_1v_1E_j\).
$$

Substituting \eqref{specr0} into \eqref{EBJ1}, we obtain
$$
b_j(\eta)=1-\frac{1}{2}\eta^2\|L^{-1}P_1v_1E_j\|+O(\eta^3).
$$

This and \eqref{EBJ} give the expansion of $\psi_j(\eta)$ for $j=2,3$, stated in \eqref{TH1}.

To obtain expansion of $\psi_j(\eta)$~$(j=-1,0,1)$ defined in \eqref{eta}, we deal with its macroscopic part and microscopic part respectively. By \eqref{XJ1}, the macroscopic part $P_0\psi_j(\eta)$~$(j=-1,0,1)$ is determined in terms of the coefficients $$\{W_0^j(\eta),W_1^j(\eta),W_2^j(\eta)\}$$ that satisfy
\bq\label{SW1}
\sigma_j(\eta) W^j_k=u_{k-1}(\eta)W^j_k+i\eta\sum^{2}_{l=0}W_l^jR_{lk}(\sigma,\eta),\quad k=0,1,2.
\eq

Furthermore, we have the normalization condition:
\bq\label{SW2}
1\equiv(\psi_j(\eta),\overline{\psi_j(\eta)})_{\eta}=W_0^j(\eta)^2+W_1^j(\eta)^2+W_2^j(\eta)^2+O(\eta^2),\quad |\eta|\leq r_0.
\eq

We expand $u_j(\eta)$ and $W_k^j(\eta)$ with $j=-1,0,1$ and $k=0,1,2$ as
$$
u_j(\eta)=u_j(0)+O(\eta^2),\quad
W_k^j(\eta)=\sum_{n=0}^1W_{k,n}^j\eta^n+O(\eta^2).
$$

Substituting the above expansion and \eqref{specr0} into \eqref{SW1} and \eqref{SW2}, we have 
\bma
O(1)&\qquad \qquad\left\{\bal
u_j(0)W_{k,0}^j=u_{k-1}(0)W^j_{k,0},\\
(W^j_{0,0})^2+(W^j_{1,0})^2+(W^j_{2,0})^2=1,
\ea\right.\label{WJ1}
\\
O(\eta)&\qquad \qquad\left\{\bal
u_j(0)W_{k,1}^j+ia_{j}W_{k,0}^j=u_{k-1}(0)W_{k,1}^j+i\sum^2_{l=0}W^j_{l,0}C_{l,k},\\
W^j_{0,0}W^j_{0,1}+W^j_{1,0}W^j_{1,1}+W^j_{2,0}W^j_{2,1}=0,
\ea\right.\label{WJ2}
\ema
where $j=-1,0,1$, $k=0,1,2$ and
$$
C_{l,n}=(L^{-1}P_1v_1E_{l-1}, v_1E_{n-1}),\quad l,n=0,1,2.
$$

By straightforward computation, we can obtain from \eqref{WJ1} and \eqref{WJ2} that
\bq\label{WJ3}
\left\{\bal
W^j_{j+1,0}=1,\quad W^j_{k,0}=0,\quad k\neq j+1,\\
W^j_{j+1,1}=0, \quad  W^j_{k,1}=\frac{iC_{j+1,k}}{u_j(0)-u_{k-1}(0)},\quad  k\neq j+1.
\ea\right.
\eq

By \eqref{eta} and \eqref{WJ3}, we obtain the expansion of $\psi_j(\eta)$~$(j=-1,0,1)$ given in \eqref{TH1}. The proof of this lemma is completed.
\end{proof}


Then, we have the decomposition of the semigroup $S(t,\eta)=e^{t\hat{B}(\eta)}$ $(S(t,\eta)=\hat{G}(t,\eta))$ as below.
\begin{lem}\label{mvpbsp4}
The semigroup $S(t,\eta)=e^{t\hat{B}(\eta)}$ has the following decomposition:
\bq
S(t,\eta)f=S_1(t,\eta)f+S_2(t,\eta)f,\quad f\in L^2(\mathbb{R}^3_v),\,\, t>0,\label{Semi0}
\eq
where
\bq
S_1(t,\eta)f=\sum^{3}_{j=-1}e^{t\lambda_j(\eta)}(f,\overline{\psi_j(\eta)})_{\eta}\psi_j(\eta)1_{\{|\eta|\leq r_0\}},
\eq
and $S_2(t,\eta)f=S(t,\eta)f-S_1(t,\eta)f$ satisfies
\bq
\|S_2(t,\eta)f\|_{\eta}\leq Ce^{-\alpha_0t}\|f\|_{\eta}, \label{Semi01}
\eq
for $\alpha_0>0$ and $C>0$ two constants independent of $\eta$.
\end{lem}

\begin{proof}
We first know that $\hat{B}(\eta)$  generates a $C_0$ contraction semigroup (cf.  \cite{LI-2}). Then, the semigroup $e^{t\hat{B}(\eta)}$ can be represented by (cf. Lemma 6.4 in \cite{LI-1})
\bq\label{BQ1}
e^{t\hat{B}(\eta)}f=\int^{\kappa+i\infty}_{\kappa-i\infty}e^{\lambda t}(\lambda-\hat{B}(\eta))^{-1}fd\lambda,\quad f\in D(\hat{B}(\eta)^2),\,\,\kappa>0.
\eq

For $\lambda\in\rho(\hat{B}(\eta))\cap\{\lambda\in \C\,|\,{\rm Re}\lambda>-\nu_0\}$, we can rewrite $(\lambda-\hat{B}(\eta))^{-1}$ as
\bq\label{BQ2}
(\lambda-\hat{B}(\eta))^{-1}=(\lambda-c(\eta))^{-1}+Z(\lambda,\eta),
\eq
where
\bma
&Z(\lambda,\eta)=(\lambda-c(\eta))^{-1}(I-Q(\eta)(\lambda-c(\eta))^{-1})^{-1}Q(\eta)(\lambda-c(\eta))^{-1},\\
&c(\eta)= -\nu(v)-iv_1\eta,\quad Q(\eta)=K-\frac{i v_1\eta  }{1+\eta^2}P^1_0. \label{cQs}
\ema

Substitute \eqref{BQ2} into \eqref{BQ1} to deduce
\bq
e^{t\hat{B}(\eta)}f=e^{tc(\eta)}f+\frac{1}{2\pi i}\int^{\kappa+i\infty}_{\kappa-i\infty}e^{\lambda t}Z(\lambda,\eta)fd\lambda.\label{Semi1}
\eq

By \eqref{cQs} and \eqref{nuv}, it holds that
\be \|e^{tc(\eta)}f\| \leq  e^{-\nu_0t}\|f\|, \quad \forall f\in L^2(\R^3_v). \label{ect}\ee

To estimate the last term on the right hand side of \eqref{Semi1}, let us denote
\bq
U_{\kappa,l}=\int^l_{-l}e^{(\kappa+iy)t}Z(\kappa+iy)fdy,
\eq
where the constant $l>0$ is chosen large enough so that $l>y_1$ with $y_1$ is given by Lemma \ref{mvpbsp0}. Let
$$
\alpha=-\frac{\mu}2, \,\ |\eta|\leq r_0;\quad \alpha=\alpha(r_0),\,\ |\eta|\geq r_0,
$$
where $\alpha(r_0)>0$ is given in Lemma \ref{mvpbsp1}. Since the operator $Z(\lambda,\eta)$ is analytic in the domain $\mathrm{Re}\lambda\geq-\alpha$ expect $\lambda=\lambda_j(\eta)$~$(j=-1,0,1,2,3)$ for $|\eta|\leq r_0$, we can shift the integration $U_{\kappa,l}$ from the line $\mathrm{Re}\lambda=\kappa$ to $\mathrm{Re}\lambda=-\alpha$ to deduce
\bq
U_{\kappa,l}=2\pi i\sum^{3}_{j=-1}\mathrm{Res}\{e^{\lambda t}Z(\lambda,\eta)f;\lambda_j(\eta)\}+U_{-\alpha,l}+H_l,
\eq
where $\mathrm{Res}\{f(\lambda);\lambda_j\}$ means the residue of $f(\lambda)$ at $\lambda=\lambda_j$ and
\bq
H_l=\(\int^{\kappa+il}_{-\alpha+il}+\int^{\kappa-il}_{-\alpha-il}\)e^{\lambda t}Z(\lambda,\eta)fd\lambda.
\eq

By Lemma 3.14 in \cite{LI-2}, we have for any $\delta>0$ and $r_0>0$ that
\bq
\left\{\bln
&\sup_{x\geq-\nu_0+\delta, |\eta|\leq r_0}\|K(x+iy-c(\eta))^{-1}\|\leq C\delta^{-\frac{3}{5}}(1+|y|)^{-\frac{2}{5}},\\
&\sup_{x\geq-\nu_0+\delta, \eta\in\R}\|(v_1\eta )(1+\eta^2)^{-1}P^1_0(x+iy-c(\eta))^{-1}\|\leq C(\delta^{-1}+1)|y|^{-1}.
\eln\right.
\eq

Thus, it is easy to verify that
\bq
\|H_l\|_{\eta}\rightarrow0,\quad as\quad l\rightarrow\infty.
\eq

Since $\lambda_j(\eta)\in\rho(c(\eta))$ and
$$
Z(\lambda,\eta)=(\lambda-\hat{B}(\eta))^{-1}-(\lambda-c(\eta))^{-1},
$$
we can prove that
\bma
\mathrm{Re}s\Big\{e^{\lambda t}Z(\lambda,\eta)f;\lambda_j(\eta)\Big\}&=\mathrm{Re}s\Big\{e^{\lambda t}(\lambda-\hat{B}(\eta))^{-1}f;\lambda_j(\eta)\Big\}\nnm\\
&=e^{\lambda_j(\eta) t}(f,\overline{\psi_j(\eta)})_{\eta}\psi_j(\eta)1_{\{|\eta|\leq r_0\}}.\label{Semi2}
\ema

To estimate the $U_{-\alpha,l}$, we have for any $f,g\in L^2_{\eta}(\mathbb{R}^3_v)$ that
\bma
&\Big|(U_{-\alpha,\infty}(t)f,g)_{\eta}\Big|\nnm\\
&\leq e^{-\alpha t}\int^{+\infty}_{-\infty}\Big|(Z(-\alpha+iy,\eta)f,g)_{\eta}\Big|dy \nnm\\
&\leq C\|Q(\eta)\|_{\eta}e^{-\alpha t}\int^{+\infty}_{-\infty} \|(-\alpha+iy-c(\eta))^{-1}f \|_{\eta} \|(-\alpha-iy-c(-\eta))^{-1}g \|_{\eta}dy \nnm\\
&\leq C\(\nu_0-\alpha\)^{-1}(\|K\|+1)e^{-\alpha t}\|f\|_{\eta}\|g\|_{\eta},
\ema
where  we have used the fact that (cf. Lemma 2.2.13 in \cite{UKAI-2})
$$
\int^{+\infty}_{-\infty}\|(x+iy-c(\eta))^{-1}f\|^2\leq \pi(x+\nu_0)^{-1}\|f\|^2,\quad x>-\nu_0.
$$

Thus
\bq
\|U_{-\alpha,\infty}(t)\|_{\eta}\leq Ce^{-\alpha t}.\label{Semi3}
\eq

Therefore,  it follows from \eqref{Semi1} and \eqref{Semi2} that
\be
e^{t\hat{B}(\eta)}f=S_1(t,\eta)f+S_2(t,\eta)f,
\ee
where
\bq
\left\{\bln
&S_1(t,\eta)f=\sum^3_{j=-1}e^{t\lambda_j(\eta)}(f,\overline{\psi_j(\eta)})_{\eta}\psi_j(\eta)1_{\{|\eta|\leq r_0\}},\\
&S_2(t,\eta)f=e^{tc(\eta)}f+U_{-\alpha,\infty}(t).
\eln\right.
\eq

In particular, $S_2(t,\eta)f$ satisfies \eqref{Semi01} in terms of \eqref{Semi3} and \eqref{ect}. The proof is completed.

\end{proof}

\section{Green's function}\setcounter{equation}{0}
\label{sect3}

In this section, we establish the pointwise space-time estimates of the  Green's function defined by \eqref{LmVPB}. First, we consider the Green's function inside Mach region $|x|\le 6t$. Based on the spectral analysis given in section \ref{sect2}, we divide the Green's function into the low-frequency fluid part, the low-frequency non-fluid part and the high-frequency part and estimate the fluid part by complex analytical techniques. To estimate the Green's function outside Mach number region $|x|> 6t$, we apply a new Picard's iteration to construct the approximate sequences which consist of the singular kinetic waves, and estimate the smooth remainder part by the weighted energy method.


\subsection{Fluid part}

In this subsection, we establish the pointwise estimates of the fluid part of  Green's function  based on the spectral analysis given in Section \ref{sect2}. To this end, we decompose the operator $G(t,x)$ into a low-frequency part and a high-frequency part:
\be \label{L-R}
\left\{\bln
&G(t,x)=G_L(t,x)+G_H(t,x),\\
&G_L(t,x)=\frac1{\sqrt{2\pi}}\int_{|\eta|<\frac{r_0}{2}} e^{ ix\eta +t\hat{B}(\eta)}d \eta,\\
&G_H(t,x)=\frac1{\sqrt{2\pi}}\int_{|\eta|>\frac{r_0}{2}} e^{ ix\eta +t\hat{B}(\eta)}d \eta,
\eln\right.
\ee
where $r_0>0$ is given by Lemma \ref{mvpbsp1}. The operator $G_L(t,x)$ can be further divided into the fluid part and the non-fluid part:
\bq\label{L-R2}
G_L(t,x)=G_{L,0}(t,x)+G_{L,1}(t,x),
\eq
where
\bma
&G_{L,0}(t,x)=\sum^{3}_{j=-1}\frac{1}{\sqrt{2\pi}}\int_{|\eta|<\frac{r_{0}}{2}}e^{ix\eta +\lambda_j(\eta)t}\psi_j(\eta)\otimes\Big\langle \psi_j(\eta)+\frac{1}{1+\eta^2}P^1_0\psi_j(\eta)\Big|d \eta,\label{GL0}\\
&G_{L,1}(t,x)=G_{L}(t,x)-G_{L,0}(t,x),
\ema
with $\lambda_j(\eta)$ and $\psi_j(\eta)$  defined by \eqref{specr0} and \eqref{specr1} respectively.
Here for any $f,g\in L^2(\mathbb{R}^3_v)$, the operator $f\otimes\langle g|$ on $L^2(\mathbb{R}^3_v)$ is defined by \cite{LIU-2,LIU-3}
$$
f\otimes\langle g|u=(u,\overline{g})f, \quad u\in L^2(\mathbb{R}^3_v).
$$

By Lemma \ref{mvpbsp4}, we have the following estimates of each part of $\hat{G}(t,\eta)$ defined by \eqref{mVPB7}.
\begin{lem}\label{mvpbgf1}
 For any $g_{0}\in L^{2}(\mathbb{R}^{3}_{v})$, there exist positive constants $C$ and $\kappa_0$ such that
\bma
&\|\hat{G}_{L}(t,\eta)g_{0}\|_{\eta}\leq C\|g_{0}\|_{\eta},\\
&\|\hat{G}_{L,1}(t,\eta)g_{0}\|_{\eta}\leq Ce^{-\frac{\mu}2t}\|g_{0}\|_{\eta},\\
&\|\hat{G}_{H}(t,\eta)g_{0}\|_{\eta}\leq Ce^{-\kappa_0t}\|g_{0}\|_{\eta},
\ema
where $\hat{G}_{L}(t,\eta)$, $\hat{G}_{L;1}(t,\eta)$, and $\hat{G}_{H}(t,\eta)$ are the Fourier transforms of $G_L(t,x)$, $G_{L,1}(t,x)$ and $G_{H}(t,x)$.
\end{lem}

With the help of Lemma \ref{mvpbsp1}, Lemma \ref{mvpbgf1} and complex analytical techniques, we have the following pointwise estimates of $G_{L,0}(t,x)$.
\begin{lem}\label{mvpbgf2}
For any given constant $D_1\geq1$, there exist three positive constants $C_{0}$, $C_{1}$ and $C$ such that for $|x|\leq D_1t$,
\bq\label{Gpd}
\left\{\bln
&\|\dxa P^j_0G_{L,0}(t,x)\|\leq C\bigg[(1+t)^{-\frac{1+\alpha}{2}}\sum^{1}_{i=-1}e^{-\frac{|x-\beta_it|^{2}}{C_{0}(1+t)}}+e^{-\frac{t}{C_{1}}}\bigg],\quad j=1,2,3,\\
&\|\dxa P_{1}G_{L,0}(t,x)\|,\|\dxa G_{L,0}(t,x)P_1\|\leq C\bigg[(1+t)^{-\frac{2+\alpha}{2}}\sum^{1}_{i=-1}e^{-\frac{|x-\beta_it|^{2}}{C_{0}(1+t)}}+e^{-\frac{t}{C_{1}}}\bigg],\\
&\|\dxa P_{1}G_{L,0}(t,x)P_1\|\leq C\bigg[(1+t)^{-\frac{3+\alpha}{2}}\sum^{1}_{i=-1}e^{-\frac{|x-\beta_it|^{2}}{C_{0}(1+t)}}+e^{-\frac{t}{C_{1}}}\bigg],
\eln\right.
\eq
where $\alpha\geq0$ and $\beta_j=j\sqrt{\frac{8}{3}}$ with $j=-1,0,1$.
\end{lem}
\begin{proof}
By \eqref{GL0}, we have
\bma
&\quad\partial^{\alpha}_xG_{L,0}(t,x)\nnm\\
&=\frac{1}{\sqrt{2\pi}}\sum^{3}_{j=-1} \int_{|\eta|<\frac{r_{0}}{2}}(i\eta)^{\alpha}e^{ix\eta +\lambda_j(\eta)t}\psi_j(\eta)\otimes\Big\langle \psi_j(\eta) +\frac{1}{1+\eta^2}P^1_0\psi_j(\eta)\Big|d\eta\nnm\\
&=:\frac{1}{\sqrt{2\pi}}\sum^{3}_{j=-1} U^\alpha_j.\label{GL1}
\ema

By \eqref{specr0} and \eqref{specr1}, it holds that for $j=-1,0,1,2,3$,
\bq
\left\{\bln
&\lambda_j(\eta)=-i\beta_{j}\eta-a_j\eta^{2}+O(\eta^{3}),\label{Lm1}\\
&\psi_j(\eta)=\psi_{j,0}+\psi_{j,1}\eta+O(\eta^{2}).
\eln\right.
\eq

For simplicity, we only deal with the term $U^\alpha_1$ and assume that $x-\beta_{1}t>0$. By Lemma \ref{mvpbsp1}, the operator $\psi_1(\eta)\otimes\langle\psi_1(\eta)+\frac{1}{1+\eta^2}P^1_0\psi_1(\eta)|$ and the function $e^{ \lambda_1(\eta)t}$ are analytic in  $\{\eta\in\mathbb{C}\,|\,|\eta|< r_0\}$. Then, we have by Cauchy Theorem that
\bma\label{I1}
U^\alpha_1&=\sum^3_{n=1}\int_{\Gamma_n}(i\eta)^{\alpha}e^{ix\eta+\lambda_1(\eta)t}\psi_1(\eta)\otimes\Big\langle\psi_1(\eta)+\frac{1}{1+\eta^2}P^1_0\psi_1(\eta)\Big|d\eta\nnm\\
&=:I^{\alpha}_{1}+I^{\alpha}_{2}+I^{\alpha}_{3},
\ema
where $\Gamma_{n}$ $(n=1,2,3)$ are given as
\bq
\left\{\bln
&\Gamma_{1}=\Big\{\eta:\ \mathrm{Re}(\eta)=-\frac{r_{0}}{2},\ 0\leq\mathrm{Im}(\eta)\leq\frac{(x-\beta_{1}t)}{C_2t}\Big\},\\
&\Gamma_{2}=\Big\{\eta: \ \mathrm{Im}(\eta)=\frac{(x-\beta_{1}t)}{C_2t},\ -\frac{r_{0}}{2}\leq\mathrm{Re}(\eta) \leq\frac{r_{0}}{2}\Big\},\\
&\Gamma_{3}=\Big\{\eta:\ \mathrm{Re}(\eta)=\frac{r_{0}}{2},\ 0\leq\mathrm{Im}(\eta)\leq\frac{(x-\beta_{1}t)}{C_2t}\Big\}.
\eln\right.
\eq

By taking $C_2=\max \{\frac{2(D_1+\beta_1)}{r_0},2a_1 \}$, we have $\Gamma_n\in \{\eta\in\mathbb{C}\,|\,|\eta|<r_0\}$ for $n=1,2,3$. Let $\eta_1=\mathrm{Re}(\eta)$ and $\eta_2=\mathrm{Im}(\eta)$.
Since $ \psi_1(\eta)= \psi_{1,0}+O(\eta)$ and $P^l_0\psi_{1,0}\neq0$, $l=1,2,3$, 
it follows that
\bma
\Big\|P^l_0I^{\alpha}_{2}\Big\|&\leq\int_{\Gamma_2}|\eta|^{\alpha}|e^{ix\eta+\lambda_1(\eta)t}|\Big\|P^l_0\psi_1(\eta)\otimes\Big\langle\psi_1(\eta)+\frac{1}{1+\eta^2}P^1_0\psi_1(\eta)\Big|\Big\|d\eta\nonumber\\
&\leq Ce^{-(1-\frac{a_{1}}{C_2})\frac{|x-\beta_{1}t|^{2}}{C_2t}}\int_{-\frac{r_{0}}{2}}^{\frac{r_{0}}{2}}e^{-\eta_1^2a_1t+O(\eta_1^3+\frac{(x-\beta_{j}t)^3}{t^3})t}\(|\eta_1|^{\alpha} +\Big|\frac{x-\beta_{1}t}{C_2t}\Big|^{\alpha}\)d\eta_1\nonumber\\
&\leq Ce^{-(1-\frac{a_{1}}{C_2})\frac{|x-\beta_{1}t|^{2}}{2C_2t}}\bigg\{\Big|\frac{x-\beta_{1}t}{C_2t}\Big|^{\alpha}\int^{\frac{r_{0}}{2}}_{-\frac{r_{0}}{2}}e^{-\frac{a_{1}\eta_1^{2}t}{2}}d\eta_1
+\int^{\frac{r_{0}}{2}}_{-\frac{r_{0}}{2}}e^{-a_{1}\eta_1^{2}t}|\eta_1|^{\alpha}d\eta_1\bigg\} \nnm\\
&\leq C(1+t)^{-\frac{1+\alpha}{2}}e^{-\frac{|x-\beta_{1}t|^2}{C_0(1+t)}},\quad l=1,2,3, \label{G2}
\ema
where $C_0,C>0$ are two constants.

Noting that $P_1\psi_1(\eta)=\eta P_1\psi_{1,1}+O(\eta^2)$, we can obtain by a similar argument as  \eqref{G2} that
\bq
\left\{\bln
&\|P_1I^{\alpha}_{2}\|,\|I^{\alpha}_{2}P_1\|\leq C(1+t)^{-\frac{2+\alpha}{2}}e^{-\frac{|x-\beta_{1}t|^2}{C_0(1+t)}},\\
&\|P_1I^{\alpha}_{2}P_1\|\leq C(1+t)^{-\frac{3+\alpha}{2}}e^{-\frac{|x-\beta_{1}t|^2}{C_0(1+t)}}.\label{GJ1}
\eln\right.
\eq

For $I^\alpha_{n}$, $n=1,3$, we have
\bma
\Big\|P^l_0I^{\alpha}_{n}\Big\|&\leq\int_{\Gamma_n}|\eta|^{\alpha}|e^{ix\eta+\lambda_1(\eta)t}|\Big\|P^l_0\psi_1(\eta)\otimes\Big\langle\psi_1(\eta)+\frac{1}{1+\eta^2}P^1_0\psi_1(\eta)\Big|\Big\|d\eta\nonumber\\
&\leq Ce^{-\frac{a_1r^2_0t}{4} +O(r_0^3)t}\int^{\frac{(x-\beta_{1}t)}{C_2t}}_{0}e^{-(x-\beta_{1}t)\eta_2+a_1\eta^2_2t}d\eta_2\nonumber\\
&\leq Ce^{-\frac{a_1r^2_0t}{8}}\int^{\frac{(x-\beta_{1}t)}{C_2t}}_{0}e^{-(1-\frac{a_{1}}{C_2})(x-\beta_{1}t)\eta_2}d\eta_2\nnm\\
&\leq Ce^{-\frac{t}{C_1}}, \quad l=1,2,3,
\ema where $C_1,C>0$ are two constants.
Similarly, it holds that for $n=1,3$,
\bq\label{G3}
\left\{\bln
&\|P_1I^{\alpha}_{n}\|,\|I^{\alpha}_{n}P_1\|\leq Ce^{-\frac{t}{C_1}},\\
&\|P_1I^{\alpha}_{n}P_1\|\leq Ce^{-\frac{t}{C_1}}.
\eln\right.
\eq

 For the case of $x-\beta_{1}t<0$, we can also obtain  \eqref{G2}--\eqref{G3} by repeating the similar argument as above.

By combining \eqref{I1} and \eqref{G2}--\eqref{G3}, we obtain
\bq\label{G4}
\left\{\bln
&\|P^l_0U^\alpha_1\|\leq C\[(1+t)^{-\frac{1+\alpha}{2}}e^{-\frac{|x-\beta_{1}t|^2}{C_0(1+t)}}+e^{-\frac{t}{C_1}}\],\quad l=1,2,3\\
&\|P_1U^\alpha_1\|+\|U^\alpha_1P_1\|\leq C\[(1+t)^{-\frac{2+\alpha}{2}}e^{-\frac{|x-\beta_{1}t|^2}{C_0(1+t)}}+e^{-\frac{t}{C_1}}\],\\
&\|P_1 U^\alpha_1 P_1\|\leq C\[(1+t)^{-\frac{3+\alpha}{2}}e^{-\frac{|x-\beta_{1}t|^2}{C_0(1+t)}}+e^{-\frac{t}{C_1}}\].
\eln\right.
\eq

For the terms $U^\alpha_j$ $(j=-1,0,2,3)$, we can repeat the above process and omit the details of proof. This, together with \eqref{GL1}, lead to \eqref{Gpd}. The proof of the lemma is completed.
\end{proof}


\subsection{Kinetic part}

In this subsection, we extract the singular kinetic part from the Green's function $G$ and establish the pointwise estimate of the remainder part. Since $\hat{G}(t,\eta)$ does not belongs to $L^1(\mathbb{R}_{\eta})$, $G(t,x)$ can be decompose into the singular part and the remainder
smooth part. Indeed, we construct the approximate sequences of $\hat{G}$ with faster decay rate in frequency space, which is equivalent to the higher regularity in physical space, and estimate the smooth remainder term by the weighted energy method. To begin with, we define the $k$-th degree Mixture operator $\mathbb{M}^t_k(\eta)$ by (cf. \cite{LIU-2,LIU-3})
\bq
\hat{\mathbb{M}}^t_k(\eta)=\int^t_0\int^{s_1}_0\cdot\cdot\cdot\int^{s_{k-1}}_0\hat{S}^{t-s_1}K\hat{S}^{s_{1}-s_2}\cdot\cdot\cdot K\hat{S}^{s_{k-1}-s_k}K\hat{S}^{s_k}ds_k\cdot\cdot\cdot ds_1,
\eq
where $\eta\in\mathbb{C}$, and $\hat{S}^t$ is a operator on $L^2(\mathbb{R}^3_v)$ defined by
$$
\hat{S}^t=e^{-[\nu(v)+iv_1\eta ]t}.
$$

Denote
$$
D_{\delta}=\{\eta\in\mathbb{C}\,|\,|\mathrm{Im}\eta|\leq\delta\}.
$$
\begin{lem}[Mixture Lemma]\label{mvpbgf3}
For any $k\geq1$, $\hat{\mathbb{M}}^t_k(\eta)$ is analytic for $\eta\in D_{\nu_0}$ and satisfies
\bma
\|\hat{\mathbb{M}}^t_{2}(\eta)g_0\| &\le C(1+|\eta|)^{-1}(1+t)^2e^{-\nu_0t}(\|g_0\| +\|\Tdv g_0\| ), \label{mix2}
\\
\|\hat{\mathbb{M}}^t_{3k}(\eta)g_0\|&\leq C_k(1+|\eta|)^{-k}(1+t)^{3k}e^{-\nu_0t}\|g_0\|  ,\label{M1}
\ema
for any $g_0\in L^2_v$ and  two positive constants $C$ and  $C_k$.
\end{lem}

\begin{proof}
This lemma is a slightly modification of the mixture lemmas in \cite{LIU-2,LI-5}. The proof is same as Lemma 3.7 in \cite{LI-5} by using the fact that
$|e^{-[\nu(v)+iv_1\eta]t}|\le e^{-\nu_0t} $ for $\eta\in D_{\nu_0}.$
\end{proof}

In terms of \eqref{B(s)}, the system \eqref{mVPB7} can be written as
\bq
\left\{\bln
&\partial_{t}\hat{G}+iv_{1}\eta \hat{G}-L\hat{G}+\frac{iv_{1}\eta }{1+\eta^2}P^1_0\hat{G}=0,\\
&\hat{G}(0,\eta)=I_{v}.\nnm
\eln\right.
\eq

We define the approximate sequence $\hat{J}_{k}$ for the Green's function $\hat{G}$ as follow:
\bq\label{26a}
\left\{\bln
&\partial_{t}\hat{J}_{0}+iv_{1}\eta \hat{J}_{0}+\nu(v)\hat{J}_{0}=0,\\
&(1+\eta^2)\hat{\Theta}_{0}=-(\hat{J}_{0},\chi_{0}),\\
&\hat{J}_{0}(0,\eta)=I_{v},
\eln\right.
\eq
and
\bq\label{26b}
\left\{\bln
&\partial_{t}\hat{J}_{k}+iv_{1}\eta \hat{J}_{k}+\nu(v)\hat{J}_{k}=K\hat{J}_{k-1}+iv_{1}\eta \chi_{0}\hat{\Theta}_{k-1},\\
&(1+\eta^2)\hat{\Theta}_{k}=-(\hat{J}_{k},\chi_{0}),\\
&\hat{J}_{k}(0,\eta)=0, \quad k\ge 1.
\eln\right.
\eq

With the help of Lemma \ref{mvpbgf3}, we can show the pointwise estimates of the approximate sequence $J_{k}(t,x)$ and $\Theta_{k}(t,x)$ as follows:
\begin{lem}\label{mvpbgf4}
For each $k\geq0$, $\hat{J}_{k}(t,\eta)$ and $\hat{\Theta}_{k}(t,\eta)$ are analytic for $\eta\in D_{\nu_0}$, and satisfy
\bq
\|\hat{J}_{3k}(t,\eta)\|+|\eta\hat{\Theta}_{3k}(t,\eta)|\leq C_k(1+|\eta|)^{-k}e^{-\frac{\nu_{0}t}{2}},\label{32}
\eq
where $C_k>0$ is a constant dependent of $k$. In particular, it holds for $k\geq2$ that
\bq
\|J_{3k}(t,x)\|+|\dx \Theta_{3k}(t,x)|\leq C_ke^{-\frac{\nu_{0}(|x|+t)}{2}}.\label{33}
\eq
\end{lem}
\begin{proof}First, we want to show \eqref{32}.
For $k=0$, it follows from \eqref{26a} that $\hat{J}_{0}(t,\eta)$ and $\hat{\Theta}_{0}(t,\eta)$ are analytic for $\eta\in D_{\nu_0}$, and satisfy
\bmas
\| \hat{J}_{0}(t,\eta)\| &= \| e^{-[\nu(v)+iv_1\eta]t}\| \le e^{-\nu_0t},
\\
| \eta\hat{\Theta}_{0}(t,\eta)|&=  \frac{|\eta|}{1+\eta^2} | (\hat{J}_{0}(t,\eta),\chi_0) |\le  \frac{1}{1+|\eta|}e^{-\nu_0t}.
\emas

Suppose that \eqref{32} holds for $k\leq l-1$. From \eqref{26b}, it holds that for $l\geq1$,
\bmas
&\hat{J}_{3l}(t,\eta)=\mathbb{M}^{t}_{3l}(\eta)+\sum^{3l-1}_{n=0}\int^{t}_{0}\mathbb{M}^{t-s}_{n}(\eta)iv_{1}\eta \chi_{0}\hat{\Theta}_{3l-n-1}(s)ds,\\
&\hat{\Theta}_{3l}(t,\eta)=-\frac{1}{1+\eta^2}(\hat{J}_{3l}(t,\eta),\chi_{0}),
\emas
with $\mathbb{M}^{t}_{0}=\hat{S}^{t}$. Then, it follows from Lemma \ref{mvpbgf3} that $\hat{J}_{3l}(t,\eta)$ and $\hat{\Theta}_{3l}(t,\eta)$ are analytic for $\eta\in D_{\nu_0}$ and satisfy
\bq\label{J1}
\|\hat{J}_{3l}(t,\eta)\|+|\eta\hat{\Theta}_{3l}(t,\eta)|\leq C(1+|\eta|)^{-l}e^{-\frac{\nu_{0}t}{2}}.
\eq

Thus,  \eqref{32} holds for any $k\ge 0$.

Next, we want to show \eqref{33}. Let $\eta=\eta_{1}+i\eta_{2}$. For $x>0$, since $\hat{J}_{k}(t,\eta)$ and $\hat{\Theta}_{k}(t,\eta)$ are analytic for $\eta\in D_{\nu_0}$, we obtain by Cauchy Theorem that
\bma
&\quad\bigg\|J_{3k}(t,x)+\dx \Theta_{3k}(t,x)\bigg\| \nnm\\
&=C\bigg\|\int^{+\infty}_{-\infty}e^{ix\eta}\[\hat{J}_{3k}(t,\eta)+i\eta\hat{\Theta}_{3k}(t,\eta)\]d\eta\bigg\|\nnm\\
&=C\bigg\|\int^{+\infty}_{-\infty}e^{ix(\eta_1+i\nu_0)}\[\hat{J}_{3k}(t,\eta_1+i\nu_0)+i(\eta_1+i\nu_0)\hat{\Theta}_{3k}(t,\eta_1+i\nu_0)\]d\eta_1\bigg\|\nnm\\
&\leq C_k\int^{+\infty}_{-\infty}e^{-\nu_0x}\(1+|\eta_1+i\nu_0|\)^{-k}e^{-\frac{\nu_{0}t}{2}}d\eta_1\nnm\\
&\leq C_ke^{-\frac{\nu_{0}(x+t)}{2}},\label{jb1}
\ema
where $k\geq2$.  Similarly, we can also obtain \eqref{jb1} for $x<0$. The proof of the lemma is completed.
\end{proof}

We define
\bma
&W_{k}(t,x)=\sum^{3k}_{i=0}J_{i}(t,x), \quad\theta_{k}(t,x)=\sum^{3k}_{i=0}\Theta_{i}(t,x),\label{28}\\
&R_{k}(t,x)=(G-W_{k})(t,x),\quad\phi_{k}(t,x)=(T-\theta_{k})(t,x),\label{29}
\ema
where $T(t,x)=(I-\partial_{xx})^{-1}(G,\chi_{0})$.

It follows from \eqref{26a}--\eqref{26b} and \eqref{28}--\eqref{29} that $W_{k}(t,x)$ and $\theta_{k}(t,x)$ satisfy
\bq\label{47}
\left\{\bln
&\partial_{t}W_{k}+v_{1}\dx W_{k}-LW_{k}-v_{1}\chi_{0}\dx \theta_{k}=-KJ_{3k}-v_{1}\chi_{0}\dx \Theta_{3k},\\
&(I-\partial_{xx})\theta_{k}=-(W_{k},\chi_{0}),\\
&W_{k}(0,x)=\delta(x)I_{v},
\eln\right.
\eq
and $R_{k}(t,x)$ and $\phi_{k}(t,x)$ satisfy
\bq\label{48}
\left\{\bln
&\partial_{t}R_{k}+v_{1}\dx R_{k}-LR_{k}-v_{1}\chi_{0}\dx \phi_{k}=KJ_{3k}+v_{1}\chi_{0}\dx \Theta_{3k},\\
&(I-\partial_{xx})\phi_{k}=-(R_{k},\chi_{0}),\\
&R_{k}(0,x)=0.
\eln\right.
\eq

With the help of Lemma \ref{mvpbgf1} and Lemma \ref{mvpbgf4}, we can show the pointwise estimate of the term $G_H(t,x)-W_{k}(t,x)$ as follow.
\begin{lem}\label{mvpbgf5}
For any given $k\geq2$, there exists a constant $C>0$ such that
\bq
\|G_H(t,x)-W_{k}(t,x)\|\leq Ce^{-\frac{\nu_2t}{4}},
\eq
where $\nu_2=\min\{\kappa_0,\frac{\nu_0}{2}\}$, and $W_k(t,x)$ is the singular kinetic wave defined by \eqref{28}.
\end{lem}

\begin{proof}
From Lemma \ref{mvpbgf1},  \eqref{32} and \eqref{48}, it holds that
\bma
\|\hat{R}_{k}(t,\eta)\|_{\eta}&=\bigg\|\int^t_0\hat{G}(t-s)[K\hat{J}_{3k}(s,\eta)+i\eta v_{1}\chi_{0}\hat{\Theta}_{3k}(s,\eta)]ds\bigg\|_{\eta}\nnm\\
&\leq C_{k}\int^t_0  (\|K\hat{J}_{3k}(s,\eta)\|_{\eta}+ | \eta\hat{\Theta}_{3k}(s,\eta) |) ds\nnm\\
&\leq C_{k}\int^t_0e^{-\frac{\nu_0s}{2}}\frac{1}{(1+|\eta|)^k}ds \leq C_{k}\frac{1}{(1+|\eta|)^k}.\label{GH1}
\ema

By \eqref{L-R} and \eqref{29}, we have
$$
\hat{G}_L(t,\eta)-\hat{R}_k(t,\eta)=\hat{W}_k(t,\eta)-\hat{G}_H(t,\eta).
$$

This together with Lemma \ref{mvpbgf1} and \eqref{GH1} implies that
\be\label{GH2}
\|\hat{W}_k(t,\eta)-\hat{G}_H(t,\eta)\|\leq \|\hat{G}_L(t,\eta)\|+\|\hat{R}_k(t,\eta)\|\leq C_k\frac{1}{(1+|\eta|)^k}.
\ee

Moreover, we obtain from Lemma \ref{mvpbgf1} and Lemma \ref{mvpbgf4} that
\be\label{GH3}
\|\hat{W}_k(t,\eta)\|+\|\hat{G}_H(t,\eta)\|\leq C_ke^{-\nu_2t},
\ee
where $\nu_2=\min\{\kappa_0,\frac{\nu_0}{2}\}$.

Thus, it follows from \eqref{GH2} and \eqref{GH3} that
$$
\|\hat{G}_H(t,\eta)-\hat{W}_k(t,\eta)\|\leq Ce^{-\frac{\nu_2t}{4}}\frac{1}{(1+|\eta|)^{\frac{3k}{4}}},
$$
which gives
\bma
\|G_H(t,x)-W_k(t,x)\|&=C\bigg\|\int_{\R}e^{ix\eta}[\hat{G}_H(t,\eta)-\hat{W}_k(t,\eta)]ds\bigg\|\nnm\\
&\leq C\int_{\R}e^{-\frac{\nu_2t}{4}}\frac{1}{(1+|\eta|)^\frac{3k}{4}}ds \leq Ce^{-\frac{\nu_2t}{4}},
\ema
for $k\geq2$. The proof is completed.
\end{proof}



\begin{lem}\label{mvpbgf6}
For any $f=f(x)\in L^2(\R_x)$ and $0<\delta<2$, it holds that
\bq\label{out2}
\int_{\mathbb{R}}|\partial^\alpha_x(I-\partial_{xx})^{-1}f|^2e^{\delta|x|}dx\leq \frac4{(2-\delta)^2}\int_{\mathbb{R}}|f|^2e^{\delta|x|}dx,
\eq
where $\alpha=0,1$.
\end{lem}
\begin{proof}
First, we know that the Green's function of $(I-\partial_{xx})^{-1}$ is $\frac{1}{2}e^{-|x|}$. Let $\alpha=0$. By H$\ddot{\mathrm{o}}$lder's inequality and Fubini's Theorem, we have
\bma
&\quad\int_{\mathbb{R}}|(I-\partial_{xx})^{-1}f|^2e^{\delta|x|}dx \nnm\\
&=\frac{1}{4}\int_{\mathbb{R}}\[\int_{\mathbb{R}}e^{-|x-y|}f(y)dy\]^2 e^{\delta|x|}dx\nnm\\
&\le \frac{1}{4}\int_{\mathbb{R}}\(\int_{\mathbb{R}}e^{-(1+\frac{\delta}2)|x-y|}e^{-\delta|y|}dye^{\delta|x|}\)\(\int_{\mathbb{R}}e^{-(1-\frac{\delta}2)|x-y|}e^{\delta|y|}|f(y)|^2dy\)dx\nnm\\
&\le \frac{1}{4}\int_{\mathbb{R}}\(\int_{\mathbb{R}}e^{-(1-\frac{\delta}2)|x-y|} dy \)\(\int_{\mathbb{R}}e^{-(1-\frac{\delta}2)|x-y|}e^{\delta|y|}|f(y)|^2dy\)dx\nnm\\
&\leq \frac1{(2-\delta)}\int_{\mathbb{R}}e^{\delta|y|}|f(y)|^2\[\int_{\mathbb{R}}e^{-(1-\frac{\delta}2)|x-y|}dx\]dy\nnm\\
&\leq \frac{4}{(2-\delta)^2}\int_{\mathbb{R}}e^{\delta|y|}|f(y)|^2dy.
\ema

For $\alpha=1$, we can obtain \eqref{out2} by using a computation similar to that of $\alpha=0$.
This proves the lemma.
\end{proof}

With the help of Lemma \ref{mvpbgf6}, we have the pointwise space-time estimates of the remainder terms $R_k(t,x)$ and $\phi_k(t,x)$  outside  Mach number region as follows.
\begin{lem}\label{mvpbgf7}
Given any constant $k\geq2$, there exist constants $C,D>0$  such that for  $|x|>6t$,
\bq
\|R_{k}(t,x)\|+| \partial_x  \phi_k(t,x)|\leq Ce^{-\frac{|x|+t}{D}},\label{53}
\eq
where $R_{k}(t,x)$ and $\phi_k(t,x)$ are defined by \eqref{29}.
\end{lem}

\begin{proof}
We use the weighted energy method. Set
\bq\label{52}
w=e^{\eps (|x|-Yt)},
\eq
where $0<\eps <1$ is a sufficient small constant and $Y>1$ is determined later. It holds that
$$
\partial_{t}w=-\eps Yw,\quad\dx w=\eps \frac{x}{|x|}w.
$$

Taking the inner product between \eqref{48} and $R_{k}w$, and integrate it over $x$, we have
\bma
&\frac12\Dt\int_{\mathbb{R}}\|R_{k}\|^{2}wdx+\frac{\eps Y}{2}\int_{\mathbb{R}}\|R_k\|^2wdx-\frac{\eps}{2} \int_{\mathbb{R}}\(\frac{x}{|x|}v_{1}R_{k},R_{k}\)wdx \nnm\\
&-\int_{\mathbb{R}}\dx \phi_{k}(R_{k}, v_{1}\chi_{0})wdx-\int_{\mathbb{R}}(LR_{k},R_{k})wdx \nnm\\
=&2\int_{\mathbb{R}}(KJ_{3k},R_{k})w dx+2\int_{\mathbb{R}}(v_{1}\chi_{0}\dx \Theta_{3k},R_{k})wdx. \label{49}
\ema

By Lemma 5.3 in \cite{LIU-2}, we have
$$
|(v_1P_0f,P_0f)|\leq \sqrt{\frac{5}{3}}(P_0f,P_0f),
$$
which leads to
\bma
&\quad\eps \int_{\mathbb{R}}\(\frac{x}{|x|}v_{1}R_{k},R_{k}\)wdx \nnm\\
&\le \eps \int_{\mathbb{R}}|(v_{1}P_0R_{k},P_0R_{k}) |wdx+2\eps \int_{\mathbb{R}}|(v_{1}P_0R_{k},P_1R_{k})| wdx\nnm\\
&\quad+\eps \int_{\mathbb{R}}|(v_{1}P_1R_{k},P_1R_{k})|wdx \nnm\\
&\leq  \eps \frac54\sqrt{\frac{5}{3}}\int_{\R}\|P_0R_k\|^2wdx+C\eps \int_{\R}(-LR_k,R_k)wdx.\label{G+}
\ema

 By partial integration, we have
\bq\label{G-}
-\int_{\mathbb{R}}\dx \phi_{k}(R_{k}, v_{1}\chi_{0})wdx=\int_{\R}\phi_k(\dx R_k,\chi_1)wdx+\int_{\R}\phi_k\(\chi_1,R_k\)\dx wdx.
\eq

Taking the inner product between $\eqref{48}_1$ and $\chi_0$, we have
\bq\label{N1}
(\partial_tR_{k},\chi_0)+(\dx R_{k},\chi_1)=(KJ_{3k},\chi_0),
\eq
which together with \eqref{G-} implies that
\be
\int_{\R}\phi_k(\dx R_k,\chi_1)wdx=\int_{\mathbb{R}}\phi_{k}(KJ_{3k},\chi_0)wdx-\int_{\mathbb{R}}\phi_{k}(\partial_tR_{k},\chi_0)wdx .\label{GRT1}
\ee

The term $\int_{\mathbb{R}}\phi_{k}(\partial_tR_{k},\chi_0)wdx$ can be estimated as follows. By $\eqref{48}_2$, we have
\bma
&\quad-\int_{\mathbb{R}}\phi_{k}(\partial_tR_{k},\chi_0)wdx\nnm\\
&=\int_{\mathbb{R}}\phi_{k}(\partial_t\phi_k-\partial_t\partial_{xx}\phi_k)wdx\nnm\\
&=\frac{1}{2}\int_{\R}\partial_t|\phi_k|^2wdx+\frac12\int_{\R} \partial_t|\dx \phi_k|^2wdx+\int_{\R}\phi_k \partial_t\dx \phi_k\dx wdx\nnm\\
&=\frac{1}{2}\Dt\int_{\R}(|\phi_k|^2+|\dx \phi_k|^2)wdx+\frac{\eps Y}{2}\int_{\R}(|\phi_k|^2+|\dx \phi_k|^2) wdx\nnm\\
&\quad+\int_{\R}\phi_k\partial_t\dx \phi_k\dx wdx. \label{47}
\ema


For the last term in the r.h.s. of \eqref{47}, we obtain by \eqref{N1} and $\eqref{48}_2$ that
$$-\partial_t\dx \phi_{k}+\pt_{xx}(I-\partial_{xx})^{-1}(R_{k},\chi_1)=\pt_x(I-\partial_{xx})^{-1}(KJ_{3k},\chi_1),$$
which together with Lemma \ref{mvpbgf6} gives rise to
\bma
&\quad\int_{\mathbb{R}}(\partial_t\dx \phi_{k})\phi_{k}\dx wdx\nnm\\
&=- \int_{\mathbb{R}}(R_{k},\chi_1)\phi_{k}\dx wdx+ \int_{\mathbb{R}}[(I-\partial_{xx})^{-1}(R_{k},\chi_1)]\phi_{k}\dx wdx \nnm\\
&\quad- \int_{\mathbb{R}}\pt_x(I-\partial_{xx})^{-1}(KJ_{3k},\chi_1) \phi_{k}\dx wdx \nnm\\
&\le \frac{\eps}2\int_{\mathbb{R}}(|(R_k,\chi_1)|^2+|(I-\partial_{xx})^{-1}(R_{k},\chi_1)|^2)wdx\nnm\\
&\quad+\frac{\eps}2\int_{\mathbb{R}} |\dx(I-\partial_{xx})^{-1}(KJ_{3k},\chi_1)|^2wdx+\frac{3\eps}{2}\int_{\mathbb{R}}|\phi_k|^2wdx\nnm\\
&\le  \frac{2\eps}{(2-\eps)^2}\int_{\mathbb{R}}(2\|R_k\|^2 +\|KJ_{3k}\|^2 )wdx+\frac{3\eps}{2}\int_{\mathbb{R}}|\phi_k|^2wdx. \label{t1}
\ema

 By \eqref{G-}--\eqref{t1}, we have
\bma
&\quad-\int_{\mathbb{R}}\dx \phi_{k}(R_{k}, v_{1}\chi_{0})wdx
\nnm\\
&\ge \frac12\Dt\int_{\R}(|\phi_k|^2+|\dx \phi_k|^2)wdx +\frac{\eps (Y-4)}{2} \int_{\R}(|\phi_k|^2+|\dx \phi_k|^2)wdx\nnm\\
&\quad-\frac{4\eps}{(2-\eps)^2} \int_{\mathbb{R}} \|R_k\|^2wdx -\frac{C}{\eps}\int_{\mathbb{R}}\|J_{3k}\|^2 wdx, \label{t2}
\ema
where we had used
$$\int_{\mathbb{R}}\phi_{k}(KJ_{3k},\chi_0)wdx \le \frac{\eps}{2}  \int_{\mathbb{R}}|\phi_{k}|^2dx+\frac{C}{\eps}\int_{\mathbb{R}}\|J_{3k}\|^2 wdx.$$

From  \eqref{G+}--\eqref{t2} and \eqref{49}, it holds that for $0<\eps\ll 1$ and $Y=5$,
\bma
&\Dt\int_{\mathbb{R}}(\|R_{k}\|^{2}+|\phi_k|^2+|\dx \phi_k|^2)wdx+\eps\int_{\mathbb{R}}(\|R_{k}\|^{2}+|\phi_k|^2+|\dx \phi_k|^2)wdx \nnm\\
&\leq \frac{C}{\eps}\int_{\mathbb{R}}(\|J_{3k}\|^{2}+|\dx \Theta_{3k}|^{2})wdx .\label{dt}
\ema
Applying Gronwall's inequality to \eqref{dt}, we have
\be
\int_{\mathbb{R}}(\|R_{k}\|^{2}+|\phi_k|^2+|\dx \phi_k|^2)wdx\le C\intt\int_{\mathbb{R}} e^{-  \eps (t-s) } (\|J_{3k}\|^{2}+|\dx \Theta_{3k}|^{2})w dx . \label{r6}
\ee

We claim that for $0\le b\le \nu_0$ and $k\ge \alpha+1$,
\be
\int_{\mathbb{R}} e^{b|x|}(\|\dxa J_{3k}\|^{2}+|\dx^{\alpha+1}\Theta_{3k}|^{2})dx\le Ce^{-\nu_0 t} . \label{kkk}
\ee
 Assume that $\hat{V}(\eta)$ is analytic for $\eta\in D_{\delta}$ with $\delta>0$ and satisfies that $|\eta|^{\alpha}|\hat{V}(\eta)|\rightarrow0$ as $|\eta|\rightarrow\infty$. Since
$$
e^{bx}\partial_x^{\alpha}V(x)=C\int^{+\infty}_{-\infty}e^{ix\eta+bx}\eta^{\alpha}\hat{V}(\eta)d\eta=C\int^{+\infty}_{-\infty}e^{ixu}(u-ib)^{\alpha}\hat{V}(u-ib)du
$$
for $x\in\R$ and $0<b<\delta$, it follows that
\bq
e^{b|x|}\partial_x^{\alpha}V(x)=
\left\{\bln
&\mathcal{F}^{-1}\((u-ib)^{\alpha}\hat{V}(u-ib)\),\ x\geq0\\
&\mathcal{F}^{-1}\((u+ib)^{\alpha}\hat{V}(u+ib)\),\ x<0.\nnm
\eln\right.
\eq
By Parseval's equality, we obtain
$$
\int_{-\infty}^{+\infty}e^{2b|x|}\|\partial_x^{\alpha}V(x)\|^2dx\leq C\int_{-\infty}^{+\infty}|u\pm ib|^{2\alpha}\|\hat{V}(u\pm ib)\|^2du.
$$
This together with Lemma \ref{mvpbgf4} implies that for $k\geq1+\alpha$ and $0\leq2\epsilon\leq\nu_0$,
\bmas
&\quad\int_{-\infty}^{+\infty}(\|\dxa J_{3k}\|^{2}+|\dx^{\alpha+1}\Theta_{3k}|^{2})e^{2\epsilon|x|}dx\nnm\\
&\leq C\int_{-\infty}^{+\infty}|u\pm i\eps|^{2\alpha}(\|\hat{J}_{3k}(t,u\pm i\eps)\|^{2}+|u\pm i\eps|^{2}|\hat{\Theta}_{3k}(t,u\pm i\eps)|^{2})du\\
&\leq C\int_{-\infty}^{+\infty}e^{-\nu_0t}(1+|u\pm i\eps|)^{{2\alpha}-2k}du\leq Ce^{-\nu_0t},
\emas
which proves \eqref{kkk}. Thus, it follows from \eqref{kkk} and \eqref{r6} that
\bq\label{50}
\int_{\mathbb{R}}(\| R_{k}\|^{2}+| \phi_k|^2+|\dx \phi_k|^2)wdx\leq C.
\eq
Similarly,
\bq\label{51}
\int_{\mathbb{R}}(\|\dx R_{k}\|^{2}+|\dx \phi_k|^2+|\dx^2 \phi_k|^2)wdx\leq C.
\eq

Thus, by \eqref{50}, \eqref{51}  and Sobolev's embedding theorem, we have
\bq\label{S2}
e^{\eps (|x|-Yt)}(\|R_k(t,x)\|+|\dx \phi_k(t,x)|)\leq C.
\eq

For $|x|>6t$,  it holds that
\be
|x|-Yt=\frac{|x|}{8}+\frac{7|x|}{8}-5t>\frac{|x|}{8}+\frac t8. \label{S3}
\ee

By \eqref{S2} and \eqref{S3}, we prove \eqref{53}.
\end{proof}

With the help of Lemma \ref{mvpbgf1}--\ref{mvpbgf2}, Lemma \ref{mvpbgf5} and Lemma \ref{mvpbgf7}, we can prove Theorem \ref{mvpbth1} as follows.

\medskip

\noindent\emph{Proof of Theorem \ref{mvpbth1}.} By \eqref{L-R}--\eqref{L-R2}, we can decompose $G(t,x)$ into
\bma
G(t,x)&=[G(t,x)-W_2(t,x)]1_{\{|x|\leq 6t\}}+[G(t,x)-W_2(t,x)]1_{\{|x|> 6t\}}+W_2(t,x)\nnm\\
&=G_{L,0}(t,x)1_{\{|x|\leq 6t\}}+[G_{L,1}(t,x)+G_{H}(t,x)-W_2(t,x)]1_{\{|x|\leq 6t\}}\nnm\\
&\quad+R_2(t,x) 1_{\{|x|> 6t\}}+W_2(t,x).
\ema

Thus
\bq
G(t,x)=G_1(t,x)+G_2(t,x)+W_2(x,v),
\eq
where
\bma
G_1(t,x)&=G_{L,0}(t,x)1_{\{|x|\leq 6t\}},\label{G1xt}\\
G_2(t,x)&=[G_{L,1}(t,x)+G_{H}(t,x)-W_2(t,x)]1_{\{|x|\leq 6t\}}+R_2(t,x) 1_{\{|x|> 6t\}},
\ema
\eqref{Thm1} directly follows from Lemma \ref{mvpbgf2}.
By Lemma \ref{mvpbgf1}, we have for $|x|\leq 6t$,
\bq\label{111j1}
\|\dxa G_{L,1}(t,x)\|\leq Ce^{-\kappa_0t}\le Ce^{-\frac{|x|+t}{D}}.
\eq

By Lemma \ref{mvpbgf5}, we have for $|x|\leq 6t$,
\bq\label{114}
\|G_{H}(t,x)-W_{2}(t,x)\|\leq Ce^{-\frac{|x|+t}{D}}.
\eq

For $|x|>6t$, it follows from Lemma \ref{mvpbgf7} that
\bq\label{112}
\|R_2(t,x)\|\leq Ce^{-\frac{|x|+t}{D}}.
\eq

By combining \eqref{111j1}--\eqref{112}, we can obtain
$$
\|G_2(t,x)\|\leq Ce^{-\frac{|x|+t}{D}},
$$
which proves  \eqref{Thm2}. This completes the proof.


\section{The Nonlinear system}\setcounter{equation}{0}
\label{sect4}

\subsection{Energy estimate}
In this subsection, we establish the energy estimate of 1-D mVPB system. First of all, let $N$ be a positive integer, and set
\bma
E_{N,k}(f)&=\sum_{\alpha+\beta\leq N}\|w^k\dx^{\alpha}\partial_{v_1}^{\beta}f\|^2_{L^2_{x,v}}+\sum_{\alpha\leq N}\|\dx^{\alpha}\Phi\|^2_{H^1_x},\label{energy1}\\
H_{N,k}(f)&=\sum_{\alpha+\beta\leq N}\|w^k\dx^{\alpha}\partial_{v_1}^{\beta}P_1f\|^2_{L^2_{x,v}}+\sum_{\alpha\leq N-1}(\|\dxa \dx P_0f\|^2_{L^2_{x,v}}+\|\dxa \dx \Phi\|^2_{H^1_x}),\label{energy2}\\
D_{N,k}(f)&=\sum_{\alpha+\beta\leq N}\|w^{k+\frac{1}{2}}\dx^{\alpha}\partial_{v_1}^{\beta}P_1f\|^2_{L^2_{x,v}}+\sum_{\alpha\leq N-1}(\|\dxa \dx P_0f\|^2_{L^2_{x,v}}+\|\dxa \dx \Phi\|^2_{H^1_x}),\label{energy3}
\ema
where $\alpha$, $\beta$ and $k$ are nonnegative integers. For the sake of brevity, we write $E_N(f)=E_{N,0}(f)$, $H_N(f)=H_{N,0}(f)$, and $D_N(f)=D_{N,0}(f)$ for $k=0$.

By taking the inner product between $\chi_j\ (j=0,1,2,3,4)$ and \eqref{mVPB8}, we obtain the compressible Euler-Poisson type system as follows:
\bma
&\partial_tn+\dx m_1=0,\label{E0j1}\\
&\partial_tm_1+\dx n+\sqrt{\frac{2}{3}}\dx q+\dx (v_1P_1f,\chi_1)=\dx \Phi+n \dx \Phi,\label{E0j2}\\
&\partial_tm_i+\dx (v_1P_1f,\chi_i)=0,\quad i=2,3,\label{E0j3}\\
&\partial_tq+\sqrt{\frac{2}{3}}\dx m_1+\dx (v_1P_1f,\chi_4)=\sqrt{\frac{2}{3}}m_1 \dx \Phi,\label{E0j5}
\ema
where
\bq
n=(f,\chi_0),\quad m_i=(f,\chi_i)\,\ (i=1,2,3),\quad q=(f,\chi_4).\label{nmq}
\eq

Taking the microscopic projection $P_1$ to \eqref{mVPB8}, we have
\bma
\partial_tP_1f+P_1(v_1\dx P_1f)-L(P_1f)=-P_1(v_1\dx P_0f)+P_1H,\label{E0j6}
\ema
where the nonlinear term $H$ is denoted by
$$
H=\frac{1}{2}(v_1\dx \Phi)f-\dx \Phi\partial_{v_1}f+\Gamma(f,f).
$$

By \eqref{E0j6}, we can express the microscopic part $P_1f$ as
\bma
P_1f=L^{-1}[\partial_tP_1f+P_1(v_1\dx P_1f)-P_1H]+L^{-1}P_1(v_1\dx P_0f).\label{E0j7}
\ema

Substituting \eqref{E0j7} into \eqref{E0j1}--\eqref{E0j5}, we obtain the the compressible Navier-Stokes-Poisson type system as
\bma
&\partial_tn+\dx m_1=0,\label{E1}\\
&\partial_tm_1+\partial_tR_{11}+\dx n+\sqrt{\frac{2}{3}}\dx q-\dx \Phi=n\dx \Phi+\frac{4}{3}\kappa_1\partial_{xx}m_1+R_{21},\label{E2}\\
&\partial_tm_i+\partial_tR_{1i}=\kappa_1\partial_{xx}m_i+R_{2i},\quad i=2,3,\label{E3}\\
&\partial_tq+\partial_tR_{14}+\sqrt{\frac{2}{3}}\dx m_1=\kappa_2\partial_{xx}q+\sqrt{\frac{2}{3}}m_1\dx \Phi+R_{24},\label{E5}
\ema
where the viscosity coefficients $\kappa_1,\ \kappa_2>0$ and the remainder term $R_{ij}$,\ $i=1,2, j=1,2,3,4$ are defined by
\bma
&\kappa_1=-(L^{-1}P_1v_1\chi_2,v_1\chi_2),\quad \kappa_2=-(L^{-1}P_1v_1\chi_4,v_1\chi_4),\nnm\\
&R_{1j}=(v_1\dx L^{-1}P_1f,\chi_j),\quad  R_{2j}=-(v_1\dx L^{-1}[P_1(v_1\dx P_1f)-P_1H],\chi_j) .\nnm
\ema

The energy estimates of the mVPB system near a global Maxwellian have been well  established when
 $x\in \R^3$, cf.  \cite{LI-2}. Note that a key feature of 3-D in the energy estimate used  in  \cite{LI-2} comes from the Sobolev's inequality
$$\|u\|_{L^6_x}\le C\|\Tdx u\|_{L^2_x},\quad \forall u=u(x),\,\, x\in \R^3.$$

However, this Sobolev's inequality does not hold  when $x\in \R$. In fact, the 1-D Sobolev's inequality takes the form
$$\|u\|_{L^\infty_x}\le C\|u\|^{1/2}_{L^2_x}\|\dx u\|^{1/2}_{L^2_x},\quad \forall u=u(x),\,\, x\in \R.$$

 By using the above Sobolev's inequality  and a similar  argument as in \cite{LI-2}, we can obtain
 the following energy estimates in 1-D setting.

\begin{lem}[Macroscopic dissipation]\label{mvpbpw3}
Let $N\geq2$ and $(n,m,q)$ be a strong solution to \eqref{E0j1}--\eqref{E0j5}. Then, there are constants $p_0>0$ and $C>0$ so that 
\bma
&\quad p_0\Dt\sum_{\alpha\leq N-1}\(\|\partial^{\alpha}_x(n,m,q)\|^2_{L^2_x}+\|\partial^{\alpha}_x\Phi\|^2_{L^2_x}+\|\partial^{\alpha}_x\dx\Phi\|^2_{L^2_x}\)\nnm\\
&\quad+2p_0\Dt\sum_{\alpha\leq N-1}\bigg(\sum^3_{i=1}\int_{\R}\partial^{\alpha}_xR_{1i}\partial^{\alpha}_xm_idx +\int_{\R}\partial^{\alpha}_xR_{14}\partial^{\alpha}_xqdx\bigg)\nnm\\
&\quad+4\Dt\sum_{\alpha\leq N-1}\int_{\R}\partial^{\alpha}_xm_1\partial^{\alpha}_x\dx ndx+\sum_{\alpha\leq N-1}\(\|\partial^{\alpha}_x\dx (n,m,q)\|^2_{L^2_x}+\|\partial^{\alpha}_x\dx \Phi\|^2_{H^1_x}\)\nnm\\
&\leq C\sum_{\alpha\leq N-1}\|\partial^{\alpha}_{x}\dx P_1f\|^2_{L^2_{x,v}}+C\sqrt{E_N(f)}D_N(f)\nnm\\
&\quad+C\sqrt{E_N(f)D_N(f)}\|P_0f\|_{L^2_v(L^{\infty}_x)}. \label{EMD}
\ema
\end{lem}
\begin{proof}
Let $\alpha$ is a non-negative integer and satisfies $\alpha\leq N-1$. Firstly,  taking the inner product between $\dxa \eqref{E2}$ and $\dxa m_1$,  we have
\bma
&\frac{1}{2}\Dt\(\|\dxa(n, m_1 )\|^2_{L^2_x} +\|\dxa \Phi\|^2_{L^2_x}+\|\dxa \dx \Phi\|^2_{L^2_x}\)+\Dt\int_{\R}\dxa R_{11}\dxa m_1dx \nnm\\
&\quad+ \sqrt{\frac{2}{3}}\int_{\R}\dxa \dx q \dxa m_1dx-\frac{4}{3}\kappa_1\|\dxa \dx m_1\|^2_{L^2_x}\nnm\\
&=\int_{\R}\dxa (n\dx \Phi)\dxa m_1dx+\int_{\R}\dxa R_{21} \dxa m_1dx+\int_{\R}\dxa R_{11}\partial_t\dxa m_1dx\nnm\\
&\quad-\int_{\R}\dxa \Phi\dxa \[(e^{-\Phi}-1)\partial_t\Phi\]dx\nnm\\
&=:I_1(\alpha)+I_2(\alpha)+I_3(\alpha)+I_4(\alpha),\label{E7}
\ema
where we have used
\bma
\int_{\R}\dxa \dx n \dxa m_1dx&=-\int_{\R}\dxa n \dxa\dx m_1dx=\frac{1}{2}\Dt\|\dxa n\|^2_{L^2_x},\label{E14j2}\\
-\int_{\R}\dxa \dx \Phi\dxa m_1 dx&=\frac{1}{2}\Dt \|\dxa (\Phi,\dx \Phi)\|^2_{L^2_x} +\int_{\R}\dxa \Phi\dxa \[(e^{-\Phi}-1)\partial_t\Phi\]dx.\nnm
\ema

We estimate $I_j(\alpha)$, $j=1,2,3,4$ as follows. For $I_1(\alpha)$, we have by Sobolev inequality that
\bma
I_1(0)&\leq \|n\|_{L^{\infty}_x}\|\dx \Phi\|_{L^2_x}\|m_1\|_{L^2_x}\leq C\sqrt{E_N(f)D_N(f)}\|n\|_{L^{\infty}_x}, \\
I_1(\alpha)&\leq C\sum_{\alpha'\leq\alpha}\|\dx^{\alpha-\alpha'}n\|_{L^2_x}\|\dx^{\alpha'}\dx \Phi\|_{L^{\infty}_x}\|\dxa m_1\|_{L^2_x}\nnm\\
&\leq C\sqrt{E_N(f)}D_N(f),\quad  1\le \alpha\le N-1.\label{E9}
\ema

For $I_2(\alpha)$, it holds that
\bma
I_2(\alpha)&\leq  C(\|\dxa \dx P_1f\|_{L^2_{x,v}}+\|\dxa (\dx \Phi f)\|_{L^2_{x,v}}+\|\nu^{-\frac{1}{2}}\dxa \Gamma(f,f)\|_{L^2_{x,v}})\| \dxa \dx m_1\|_{L^2_x}\nnm\\
&\leq \frac{\kappa_1}{2}\|\dxa \dx m_1\|^2_{L^2_x}+C\|\dxa \dx P_1f\|^2_{L^2_{x,v}}+C\sqrt{E_N(f)}D_N(f)\nnm\\
&\quad+C\sqrt{E_N(f)D_N(f)}\|P_0f\|_{L^2_v(L^{\infty}_x)},\label{E10}
\ema
where we have used
\bma
&\quad \|\nu^{-\frac{1}{2}}\Gamma(f,f)\|_{L^2_{x,v}}\nnm\\
 &\leq\|\nu^{-\frac{1}{2}}\Gamma(P_0f,P_0f)\|_{L^2_{x,v}}+\|\nu^{-\frac{1}{2}}\Gamma(P_1f,P_1f)\|_{L^2_{x,v}}
 +\|\nu^{-\frac{1}{2}}\Gamma(P_0f,P_1f)\|_{L^2_{x,v}}\nnm\\
&\leq C\sqrt{E_N(f)D_N(f)}+C\sqrt{E_N(f)}\|P_0f\|_{L^2_v(L^{\infty}_x)}, \label{GA1}
\ema
and
\bma
 &\quad\|\nu^{-\frac{1}{2}}\dxa \Gamma(f,f)\|_{L^2_{x,v}}\nnm\\
 &\leq C\sum_{1\le \alpha'\leq\alpha-1}\Big(\|\dx^{\alpha'}f\|_{L^2_v(L^{\infty}_x)}\|\nu^{\frac{1}{2}}\dx^{\alpha-\alpha'}f\|_{L^2_{x,v}}
 +\|\nu^{\frac{1}{2}}\dx^{\alpha'}f\|_{L^2_v(L^{\infty}_x)}\|\dx^{\alpha-\alpha'}f\|_{L^2_{x,v}}\Big)\nnm\\
&\quad+C\Big(\|\dxa f\|_{L^2_{x,v}}\|\nu^{\frac{1}{2}}f\|_{L^2_v(L^{\infty}_x)}+\|\nu^{\frac{1}{2}}\dxa f\|_{L^2_{x,v}}\|f\|_{L^2_v(L^{\infty}_x)}\Big)\nnm\\
&\leq C\sqrt{E_{N}(f)D_N(f)},\quad 1\leq\alpha\leq N.\label{GA2}
\ema

For $I_3(\alpha)$, we have by Cauchy inequality and \eqref{E0j2} that
\bma
I_3(\alpha)&=\int_{\R}\dxa R_{11}\dxa \bigg[-\dx n-\sqrt{\frac{2}{3}}\dx q-\dx (v_1P_1f,\chi_1)+\dx \Phi+n\dx \Phi\bigg]dx\nnm\\
&\leq \frac{C}{\eps}\|\dxa \dx P_1f\|^2_{L^2_{x,v}}+\eps\(\|\dxa \dx n\|^2_{L^2_x}+\|\dxa \dx q\|^2_{L^2_x}+\|\dxa \dx \Phi\|^2_{L^2_x}\)\nnm\\
&\quad+C\sqrt{E_N(f)}D_N(f),\label{E11}
\ema
where $0<\eps<1$ is a constant determined later.

By \eqref{mVPB9} and \eqref{E0j1}, we obtain
\be
\partial_t\Phi-\partial_{xx}\partial_t\Phi =\dx m_1-(e^{-\Phi}-1)\partial_t\Phi.\label{E12}
\ee

Then, when $E_N(f)$ is small enough, we take the inner product of $\dxa \partial_t\Phi$ and $\dxa \eqref{E12}$ to get
\be
\sum_{\alpha\leq N-1}(\|\dxa \partial_t\Phi\|^2_{L^2_x}+\|\dxa \partial_t\dx \Phi\|^2_{L^2_x})\leq C\sum_{\alpha\leq N-1}\|\dxa \dx m_1\|^2_{L^2_x}.\label{E13}
\ee

Noting that
$$
e^{-\Phi}-1=n-\partial_{xx}\Phi,
$$
we obtain by \eqref{E13} that
\bma
I_4(0)
&\leq C\|\Phi\|_{L^2_x}\|\partial_t\Phi\|_{L^2_x}\|n\|_{L^{\infty}_x}+C\|\Phi\|_{L^2_x}\|\partial_t\Phi\|_{L^2_x}\|\partial_{xx}\Phi\|_{L^{\infty}_x}\nnm\\
&\leq C\sqrt{E_N(f)}D_N(f)+C\sqrt{E_N(f)D_N(f)}\|n\|_{L^{\infty}_x},  \label{E14}
\\
I_4(\alpha)&\leq \sum_{\alpha'\leq\alpha}\|\dxa \Phi\|_{L^2_x}\|\dx^{\alpha'}(e^{-\Phi}-1)\|_{L^{\infty}_x}\|\dx^{\alpha-\alpha'}\partial_t\Phi\|_{L^2_x}\nnm\\
&\leq C\sqrt{E_N(f)}D_N(f), \quad 1\le \alpha\le N-1.\label{E14j1}
\ema

Therefore, it follows from \eqref{E7}--\eqref{E10},and \eqref{E11}--\eqref{E14j1} that for $0<\eps\ll 1$,
\bma
&\quad\frac{1}{2}\Dt\(\|\dxa(n, m_1)\|^2_{L^2_x} +\|\dxa \Phi\|^2_{L^2_x}+\|\dxa \dx \Phi\|^2_{L^2_x}+2\int_{\R}\dxa R_{11}\dxa m_1dx\)\nnm\\
&\quad+\sqrt{\frac{2}{3}}\int_{\R}\dxa \dx q\dxa  m_1dx +\frac{5}{6}\kappa_1\|\dxa \dx m_1\|^2_{L^2_x}\nnm\\
&\leq C\|\dxa \dx P_1f\|^2_{L^2_{x,v}}+\eps\(\|\dxa \dx n\|^2_{L^2_x}+\|\dxa \dx q\|^2_{L^2_x}+\|\dxa \dx \Phi\|^2_{L^2_x}\)\nnm\\
&\quad+C\sqrt{E_N(f)}D_N(f)+C\sqrt{E_N(f)D_N(f)}\|P_0f\|_{L^2_v(L^{\infty}_x)}.\label{E15}
\ema

 By taking the inner product between $\dxa m_i$ $(i=2,3)$ and $\dxa \eqref{E3}$, we can obtain by a similar argument as proving \eqref{E15} that
\bma
&\frac{1}{2}\Dt\(\|\dxa m_i\|^2_{L^2_x}+2\int_{\R}\dxa R_{1i}\dxa m_idx\)+\frac{\kappa_1}{2}\|\dxa \dx m_i\|^2_{L^2_x}\nnm\\
&\leq C\|\dxa \dx P_1f\|^2_{L^2_{x,v}}+C\sqrt{E_N(f)}D_N(f)+C\sqrt{E_N(f)D_N(f)}\|P_0f\|_{L^2_v(L^{\infty}_x)}.\label{E19}
\ema

Similarly, we have
\bma
&\quad\frac{1}{2}\Dt\(\|\dxa q\|^2_{L^2_x}+2\int_{\R}\dxa R_{14}\dxa qdx\)+\sqrt{\frac{2}{3}}\int_{\R}\dxa \dx m_1 \dxa qdx+\frac{\kappa_2}{2}\|\dxa \dx q\|^2_{L^2_x}\nnm\\
&\leq C\|\dxa \dx P_1f\|^2_{L^2_{x,v}}+\eps\|\dxa \dx m_1\|^2_{L^2_x}+C\sqrt{E_N(f)}D_N(f)\nnm\\
&\quad+C\sqrt{E_N(f)D_N(f)}\|P_0f\|_{L^2_v(L^{\infty}_x)}.\label{E25}
\ema

Finally, taking the inner product between $\dxa \dx n$ and $\dxa \eqref{E0j2}$ with $\alpha\leq N-1$ and applying a similar argument as proving \eqref{E15}, we have
\bma
&\quad\Dt\int_{\R}\dxa m_1\dxa \dx ndx+\frac{1}{2}\|\dxa \dx n\|^2_{L^2_x}+\|\dxa \dx \Phi\|^2_{L^2_{x}}+\|\dxa \partial_{xx}\Phi\|^2_{L^2_x}\nnm\\
&\leq \|\dxa \dx m_1\|^2_{L^2_x}+C(\|\dxa \dx q\|^2_{L^2_x}+\|\dxa \dx P_1f\|^2_{L^2_{x,v}})+C\sqrt{E_N(f)}D_N(f).\label{E26}
\ema

Making the summation $p_0\sum_{\alpha\leq N-1}[\eqref{E15}+ \eqref{E19}+\eqref{E25}]+4\sum_{\alpha\leq N-1}\eqref{E26}$ with the constant $p_0>0$ large enough and $\eps>0$ small enough, we obtain \eqref{EMD}. The proof of the lemma is completed.
\end{proof}


In the following, we can estimate the microscopic terms to enclose the energy inequality of the solution $f$ to mVPB system \eqref{mVPB8}--\eqref{mVPB9}.

\begin{lem}\label{mvpbpw4}
Let $N\geq2$ and $(f,\Phi)$ be a strong solution to 1-D mVPB system \eqref{mVPB8}--\eqref{mVPB9}. Then, there are constants $p_k>0,\ 1\leq k\leq N$ and $C>0 $  such that
\bma
&\frac{1}{2}\Dt\sum_{\alpha\leq N}(\|\partial^{\alpha}_{x}f\|^2_{L^2_{x,v}}+\|\partial^{\alpha}_{x}\Phi\|^2_{L^2_x}+\|\partial^{\alpha}_{x}\dx\Phi\|^2_{L^2_x})
+\mu\sum_{\alpha\leq N}\|w^{\frac{1}{2}}\partial^{\alpha}_xP_1f\|^2_{L^2_{x,v}}\nnm\\
&\leq C\sqrt{E_N(f)}D_N(f)+C\sqrt{E_N(f)D_N(f)}\|P_0f\|_{L^2_v(L^{\infty}_x)},\label{MI1}\\
&\Dt\sum_{1\leq k\leq N}p_k\sum_{\beta=k, \alpha+\beta\leq N}\|\dxa \partial^{\beta}_{v_1}P_1f\|^2_{L^2_{x,v}}\nnm\\
&\quad+\mu\sum_{1\leq k\leq N}p_k\sum_{\beta=k,\alpha+\beta\leq N}\|w^{\frac{1}{2}}\dxa \partial^{\beta}_{v_1}P_1f\|^2_{L^2_{x,v}}\nnm\\
&\leq C\sum_{\alpha\leq N-1}\|\dxa \dx f\|^2_{L^2_{x,v}}+C\sqrt{E_N(f)}D_N(f)+C\sqrt{E_N(f)D_N(f)}\|P_0f\|_{L^2_v(L^{\infty}_x)}.\label{MI3}
\ema
\end{lem}

\begin{proof}
Taking the inner product between $\partial^{\alpha}_xf$ and $\dxa \eqref{mVPB8}$, we have
\bma
&\frac{1}{2}\Dt(\|\partial^\alpha_xf\|^2_{L^2_{x,v}}+\|\dxa\Phi\|^2_{L^2_x}+\|\dxa \dx  \Phi\|^2_{L^2_x})-\int_{\R\times\R^3}L(\partial^{\alpha}_xf)\partial^{\alpha}_xfdxdv\nnm\\
&=\frac12\int_{\R\times\R^3}v_{1}\partial^{\alpha}_x(\dx \Phi f)\partial^{\alpha}_xfdxdv-\int_{\R\times\R^3}\partial^{\alpha}_x(\dx  \Phi\partial_{v_{1}} f)\partial^{\alpha}_xfdxdv\nnm\\
&\quad+\int_{\R\times\R^3}\partial^{\alpha}_x\Gamma(f,f)\partial^{\alpha}_xP_1fdxdv+\int_{\R\times\R^3}\dxa \Phi\dxa [\partial_t\Phi(1-e^{-\Phi})] dxdv\nnm\\
&=:J_1(\alpha)+J_2(\alpha)+J_3(\alpha)+J_4(\alpha).\label{IE15}
\ema

We estimate $J_j(\alpha)$, $j=1,2,3,4$ as follows. For $J_1(\alpha)$, we have by Sobolev inequality that
\bma
J_1(0)&\leq C\|\dx \Phi\|_{L^2_x}\|\nu^{\frac{1}{2}}f\|_{L^2_v(L^{\infty}_{x})}\|\nu^{\frac{1}{2}}f\|_{L^2_{x,v}}\nnm\\
&\leq C\|\dx \Phi\|_{L^2_x}\|\nu^{\frac{1}{2}}f\|_{L^2_{x,v}}(\|\nu^{\frac{1}{2}}P_0f\|_{L^2_v(L^{\infty}_{x})}+\|\nu^{\frac{1}{2}}P_1f\|_{L^2_v(L^{\infty}_{x})})\nnm\\
&\leq C\sqrt{E_N(f)}D_N(f)+C\sqrt{E_N(f)D_N(f)}\|P_0f\|_{L^2_v(L^{\infty}_x)},\label{IE16j}\\
J_1(\alpha)&\leq C\sum_{\alpha'\leq \alpha-1}\|\dx^{\alpha'}\dx \Phi\|_{L^{\infty}_x}\|\nu^{\frac{1}{2}}\dx^{\alpha-\alpha'}f\|_{L^2_{x,v}}\|\nu^{\frac{1}{2}}\dxa f\|_{L^2_{x,v}}\nnm\\
&\quad+C \|\dxa \dx \Phi\|_{L^{2}_x}\|\nu^{\frac{1}{2}} f\|_{L^2_{v}(L^{\infty}_x)}\|\nu^{\frac{1}{2}}\dxa f\|_{L^2_{x,v}}\nnm\\
&\leq C\sqrt{E_N(f)}D_N(f),\quad 1\leq\alpha\leq N.\label{IE16}
\ema

For $J_2(\alpha)$, we have by Sobolev inequality that
\bma
J_2(0)&=\frac{1}{2}\int_{\R\times\R^3}\dx \Phi\partial_{v_1}(f)^2dxdv=0,\label{IE17j}\\
J_2(\alpha)&=\sum_{1\le \alpha'\le \alpha}\int_{\R\times\R^3}\dx^{\alpha'}\Phi\dx^{\alpha-\alpha'}\partial_{v_1}f\dxa fdxdv\nnm\\
&\leq C\sum_{1\leq\alpha'\leq\alpha-1}\|\partial^{\alpha'}_x\dx \Phi\|_{L^{\infty}_x}\|\partial^{\alpha-\alpha'}_x\partial_{v_1}f\|_{L^2_{x,v}}\|\partial^{\alpha}_xf\|_{L^2_{x,v}}\nnm\\
&\quad+C \|\dxa\dx \Phi\|_{L^2_x}\| \partial_{v_1}f\|_{L^2_{v}(L^{\infty}_x)}\|\partial^{\alpha}_xf\|_{L^2_{x,v}}\nnm\\
&\leq C\sqrt{E_N(f)}D_N(f),\quad 1\leq\alpha\leq N.\label{IE17}
\ema

For $J_3(\alpha)$, we have by \eqref{GA1}--\eqref{GA2} that
\bma
J_3(\alpha)&\leq \|\nu^{-\frac{1}{2}}\dxa \Gamma(f,f)\|_{L^2_{x,v}}\|\nu^{\frac{1}{2}}\dxa P_1f\|_{L^2_{x,v}}\nnm\\
&\leq C\sqrt{E_N(f)}D_N(f)+C\sqrt{E_N(f)D_N(f)}\|P_0f\|_{L^2_v(L^{\infty}_x)}.\label{IE18}
\ema

For $J_4(\alpha)$, we obtain by \eqref{E13} that
\bma
J_4(\alpha)&\leq \sum_{\alpha'\leq\alpha}\|\dxa \Phi\|_{L^2_x}\|\dx^{\alpha'}(e^{-\Phi}-1)\|_{L^{\infty}_x}\|\partial^{\alpha-\alpha'}_{x}\dt\Phi\|_{L^2_x}\nnm\\
&\leq C\sqrt{E_N(f)}D_N(f). \label{IE19j5}
\ema

By combining \eqref{IE15}--\eqref{IE19j5}, we can obtain
\bma
&\quad\frac{1}{2}\Dt(\|\partial^\alpha_xf\|^2_{L^2_{x,v}}+\|\dxa \Phi\|^2_{L^2_x}+\|\dxa \dx \Phi\|^2_{L^2_x})+\mu\|w^{\frac{1}{2}}\partial^{\alpha}_xP_1f\|^2_{L^2_{x,v}}\nnm\\
&\leq C\sqrt{E_N(f)}D_N(f)+C\sqrt{E_N(f)D_N(f)}\|P_0f\|_{L^2_v(L^{\infty}_x)}.\label{IE20}
\ema

By taking the summation $\sum_{\alpha\le N}\eqref{IE20}$, we can obtain \eqref{MI1}.

In order to close the energy estimate, we need to estimate the term $\partial_{v_1}f$. For this, we rewrite \eqref{mVPB8} as
\bma
&\quad\partial_tP_1f+v_1\dx P_1f+\dx \Phi\partial_{v_1}P_1f-L(P_1f)\nnm\\
&=\Gamma(f,f)+\frac{1}{2}v_1\dx \Phi P_1f +P_0\(v_1\dx P_1f+\dx \Phi\partial_{v_1}P_1f-\frac{1}{2}v_1\dx \Phi P_1f\)\nnm\\
&\quad+P_1\(\frac{1}{2}v_1\dx \Phi P_0f-\dx \Phi\partial_{v_1}P_0f-v_1\dx P_0f\).\label{IE7}
\ema

Next, let $1\leq k\leq N$, and choose positive constants $\alpha,\beta$ with $\beta=k$ and $\alpha+\beta\leq N$. Taking the inner product between $\dxa \partial^{\beta}_{v_1}P_1f$ and $\dxa \partial^{\beta}_{v_1}\eqref{IE7}$, we have
\bma
&\quad\frac{1}{2}\Dt\|\dxa \partial^{\beta}_{v_1}P_1f\|^2_{L^2_{x,v}}+\frac{1}{2}\mu\|w^{\frac{1}{2}}\dxa \partial^{\beta}_{v_1}P_1f\|^2_{L^2_{x,v}}\nnm\\
&\leq C(\|\partial^{\alpha}_x\dx P_0f\|^2_{L^2_{x,v}}+\|\partial^{\alpha}_x\dx P_1f\|^2_{L^2_{x,v}})+C\sum_{\beta'=k-1}\|\dxa \partial^{\beta'}_{v_1}\dx P_1f\|^2_{L^2_{x,v}}\nnm\\
&\quad+C\sum_{\beta'\leq k-1}\|\dxa \partial^{\beta'}_{v_1}P_1f\|^2_{L^2_{x,v}}+C\sqrt{E_N(f)}D_N(f)\nnm\\
&\quad+C\sqrt{E_N(f)D_N(f)}\|P_0f\|_{L^2_v(L^{\infty}_x)}.\label{IE26}
\ema

Thus, by taking summation of \eqref{IE26} over $\{\beta=k,\ \alpha+\beta\leq N\}$, we obtain
\bma
&\Dt\sum_{\beta=k,\alpha+\beta\leq N}\|\dxa \partial^{\beta}_{v_1}P_1f\|^2_{L^2_{x,v}}+\mu\sum_{\beta=k,\alpha+\beta\leq N}\|w^{\frac{1}{2}}\dxa \partial^{\beta}_{v_1}P_1f\|^2_{L^2_{x,v}}\nnm\\
&\leq C\sum_{\alpha\leq N-k}(\|\partial^{\alpha}_x\dx P_0f\|^2_{L^2_{x,v}}+\|\partial^{\alpha}_x\dx P_1f\|^2_{L^2_{x,v}})+C_k\sum_{\beta=k-1,\alpha+\beta\leq N}\|\partial^{\alpha}_x\partial_{v_1}^{\beta}P_1f\|^2_{L^2_{x,v}}\nnm\\
&\quad+C\sqrt{E_N(f)}D_N(f)+C\sqrt{E_N(f)D_N(f)}\|P_0f\|_{L^2_v(L^{\infty}_x)}.\label{IE27}
\ema

Then, taking summation $\sum_{1\leq k\leq N}p_k\eqref{IE27}$ with constants $p_k>0$ chosen by
$$
\mu p_k\geq2\sum_{1\leq j\leq N-k}p_{k+j}C_{k+j},\quad 1\leq k\leq N-1,\,\ p_N=1,
$$
we obtain \eqref{MI3}. The proof of the Lemma is completed.
\end{proof}


With the help of Lemmas \ref{mvpbpw3}--\ref{mvpbpw4}, we have the following energy estimate.
\begin{lem}\label{mvpbpw5}
For $N\geq2$, there are two equivalent energy functionals $\mathcal{E}_{N}(\cdot)\sim E_{N}(\cdot)$ and $\mathcal{H}_{N}(\cdot)\sim H_{N}(\cdot)$ such that if $E_N(f)$ is sufficiently small, then the solution $f(t,x,v)$ to the 1-D mVPB system \eqref{mVPB8}--\eqref{mVPB9} satisfies
\bma
&\Dt\mathcal{E}_{N}(f)(t)+ D_{N}(f)(t)\leq CE_N(f)\|P_0f\|^2_{L^2_v(L^{\infty}_x)},\label{ENI1j1}\\
&\Dt\mathcal{H}_{N}(f)(t)+ D_{N}(f)(t)\leq C\|\dx P_0f\|^2_{L^2_{x,v}}+CE_N(f)\|P_0f\|^2_{L^2_v(L^{\infty}_x)}.\label{ENI2j2}
\ema
\end{lem}
\begin{proof}
Assume that
$$
E(f)(t)\leq \delta
$$
for $\delta>0$ small. Taking the summation of $A_1\eqref{EMD}+A_2\eqref{MI1}+\eqref{MI3}$ with $A_2\mu>CA_1>0$ large enough, we have
\bma
&A_1p_0\Dt\sum_{\alpha\leq N-1}(\|\dxa (n,m,q)\|^2_{L^2_{x,v}}+\|\dxa \Phi\|^2_{H^1_x})+A_2\Dt\sum_{\alpha\leq N}(\|\dxa f\|^2_{L^2_{x,v}}+\|\dxa \Phi\|^2_{H^1_x})\nnm\\
&+2A_1p_0\Dt\sum_{\alpha\leq N-1}\(\sum^3_{i=1}\int_{\R}\partial^{\alpha}_xR_{1i}\partial^{\alpha}_xm_idx+\int_{\R}\dxa R_{14}\dxa qdx\)\nnm\\
&+4A_1\Dt\sum_{\alpha\leq N-1}\int_{\R}\dxa m_1\dxa \partial_xndx+\Dt\sum_{1\leq k\leq N}p_k\sum_{\beta=k, \alpha+\beta\leq N}\|\dxa \partial^{\beta}_{v_1}P_1f\|^2_{L^2_{x,v}}\nnm\\
&+A_2\mu\sum_{\alpha\leq N}\|w^{\frac{1}{2}}\dxa P_1f\|^2_{L^2_{x,v}}+\mu\sum_{\alpha+\beta\leq N}\|w^{\frac{1}{2}}\dxa \partial^{\beta}_{v_1}P_1f\|^2_{L^2_{x,v}}\nnm\\
&+\sum_{\alpha\leq N-1}(\|\dxa \partial_x(n,m,q)\|^2_{L^2_x}+\|\dxa \partial_x\Phi\|^2_{H^1_x})\leq CE_N(f)\|P_0f\|^2_{L^2_v(L^{\infty}_x)},\label{EI2}
\ema
which implies \eqref{ENI1j1}.

By taking the inner product between $P_1f$ and \eqref{IE7}, we have
\bma
&\frac{1}{2}\Dt\|P_1f\|^2_{L^2_{x,v}}+\frac{1}{2}\mu\|w^{\frac{1}{2}}P_1f\|^2_{L^2_{x,v}} \nnm\\
&\leq C\|\dx P_0f\|^2_{L^2_{x,v}}+C\sqrt{E_N(f)}D_N(f)+C\sqrt{E_N(f)D_N(f)}\|P_0f\|_{L^2_v(L^{\infty}_x)}.\label{IE13}
\ema

Taking the summation of
$$
A_1p_0\sum_{1\leq\alpha\leq N}[\eqref{E15}+\eqref{E19}+\eqref{E25}]+A_1\sum_{\alpha\leq N}\eqref{E26}+A_2\sum_{1\leq\alpha\leq N}\eqref{IE20}+\eqref{MI3}+\eqref{IE13}
$$
with $A_2\mu>CA_1>0$ large enough, we have
\bma
&A_1p_0\Dt\sum_{1\leq\alpha\leq N-1}(\|\dxa (n,m,q)\|^2_{L^2_{x,v}}+\|\dxa \Phi\|^2_{H^1_x})\nnm\\
&+2A_1p_0\Dt\sum_{1\leq\alpha\leq N-1}\(\sum^3_{i=1}\int_{\R}\partial^{\alpha}_xR_{1i}\partial^{\alpha}_xm_idx+\int_{\R}\dxa R_{14}\dxa qdx\)\nnm\\
&+4A_1\Dt\sum_{\alpha\leq N-1}\int_{\R}\dxa m_1\dxa \partial_xndx+A_2\Dt\sum_{1\leq\alpha\leq N}(\|\dxa f\|^2_{L^2_{x,v}}+\|\dxa \Phi\|^2_{H^1_x})\nnm\\
&+\Dt\|P_1f\|^2_{L^2_{x,v}}+\Dt\sum_{1\leq k\leq N}p_k\sum_{\beta=k,\alpha+\beta\leq N}\|\partial^{\alpha}_x\partial^{\beta}_{v_1}P_1f\|^2_{L^2_{x,v}}+\mu\sum_{\alpha\leq N}\|w^{\frac{1}{2}}\dxa P_1f\|^2_{L^2_{x,v}}\nnm\\
&+\mu\sum_{\alpha+\beta\leq N}\|w^{\frac{1}{2}}\dxa \partial^{\beta}_{v_1}P_1f\|^2_{L^2_{x,v}}+A_1\sum_{\alpha\leq N-1}(\|\dxa \partial_x(n,m,q)\|^2_{L^2_x}+\|\dxa \partial_x\Phi\|^2_{H^1_x})\nnm\\
&\leq C\|\dx P_0f\|^2_{L^2_{x,v}}+CE_N(f)\|P_0f\|^2_{L^2_v(L^{\infty}_x)}.\label{EI3}
\ema
which proves \eqref{ENI2j2}. The proof is completed.
\end{proof}

Repeating the proof of Lemmas \ref{mvpbpw3}--\ref{mvpbpw5}, we have the following weighted energy estimate.
\begin{lem}\label{mvpbpw6}
For $N\geq2$ and $k\ge 1$, there are two equivalent energy functionals $\mathcal{E}_{N,k}(\cdot)\sim E_{N,k}(\cdot)$ and $\mathcal{H}_{N,k}(\cdot)\sim H_{N,k}(\cdot)$ such that if $E_{N,k}(f)$ is sufficiently small, then the solution $f(t,x,v)$ to the 1-D mVPB system \eqref{mVPB8}--\eqref{mVPB9} satisfies
\bma
&\Dt\mathcal{E}_{N,k}(f)(t)+ D_{N,k}(f)(t)\leq CE_N(f)\|P_0f\|^2_{L^2_v(L^{\infty}_x)},\label{ENI1}\\
&\Dt\mathcal{H}_{N,k}(f)(t)+ D_{N,k}(f)(t)\leq C\|\dx P_0f\|^2_{L^2_{x,v}}+CE_N(f)\|P_0f\|^2_{L^2_v(L^{\infty}_x)}.\label{ENI2}
\ema
\end{lem}

%
%

\subsection{The pointwise estimate}

In this subsection, we prove Theorem \ref{mvpbth2} on the pointwise behaviors of the global solution to the nonlinear mVPB system with the help of the estimates of the Green's function given in Section \ref{sect3}.


First, we  give some basic estimates  of convolution of the initial data and different waves in order to analyze the pointwise behaviors of the solution. Indeed, we have

\begin{lem}[\cite{LI-6}]\label{mvpbpw8}
For any given $\lambda\in\R,\ \alpha\geq0,\ \gamma_0>\frac{1}{2}$, there exist a constant $C>0$ such that
\bma
\int_{\mathbb{R}}\frac{e^{-\frac{|x-y-\lambda t|^2}{D(1+t)}}}{(1+t)^{\alpha}}(1+|y|^2)^{-\gamma_0} dy\leq C(1+t)^{-\alpha}B_{\gamma_0}(t,x-\lambda t).\label{CI}
\ema
\end{lem}

Set
\bma
&\quad I^{\alpha,\beta,\gamma}(t,x;t_1,t_2;\lambda,\mu,D)\nnm\\
&=\int^{t_2}_{t_1}\int_{\R}(1+t-s)^{-\alpha}e^{-\frac{|x-y-\lambda(t-s)|^2}{D(1+t-s)}}(1+s)^{-\beta}B_{\gamma}(y-\mu s,s)dyds,\label{CAI1}\\
&\quad J^{\beta,\gamma}(t,x;t_1,t_2;\lambda,D)\nnm\\
&=\int^{t_2}_{t_1}\int_{\R}e^{-\frac{|x-y|+t-s}{D}}(1+s)^{-\beta}B_{\gamma}(s,y-\lambda s)dyds.\label{CJ0}
\ema

\begin{lem}[\cite{LI-6}]\label{mvpbpw9}
For any given $\beta,\gamma\ge 1$, $\lambda,\mu\in \R$, $D>0$ and  $\gamma>\frac{1}{2}$, there exists  a constant $C>0$ such that the following holds. \\
(1) For $\lambda=\mu$,
\bma
&I^{\alpha,\beta,\gamma}\(t,x;0,\frac{t}{2};\lambda,\mu,D\)\leq C(1+t)^{-\alpha}B_{\gamma}(t,x-\lambda t)\mathrm{T}^{\beta-\frac{1}{2}}(t),\label{nlj1}\\
&I^{\alpha,\beta,\gamma}\(t,x;\frac{t}{2},t;\lambda,\mu,D\)\leq C(1+t)^{-\beta}B_{\gamma}(t,x-\lambda t)\mathrm{T}^{\alpha-\frac{1}{2}}(t).\label{nlj2}
\ema
(2) For $\lambda\neq\mu$,
\bma
&I^{\alpha,\beta,\gamma}\(t,x;0,\frac{t}{2};\lambda,\mu,D\)\nnm\\
&\leq C(1+t)^{-\alpha}B_{\gamma}(t,x-\lambda t)\mathrm{T}^{\beta-\frac{1}{2}}(t)\nonumber\\
&\quad+C(1+t)^{-\alpha+\frac{1}{2}}(1+x-\lambda t)^{-\beta+\frac{1}{2}}1_{\lambda t+\sqrt{1+t}\leq x\leq\mu t-\sqrt{1+t}},\label{45j1}\\
&I^{\alpha,\beta,\gamma}\(t,x;\frac{t}{2},t;\lambda,\mu,D\)\nnm\\
&\leq C(1+t)^{-\beta}B_{\gamma}(x-\mu t,t)\mathrm{T}^{\alpha-\frac{1}{2}}(t)\nonumber\\
&\quad+C(1+t)^{-\beta+\frac{1}{2}}(1+\mu t-x)^{-\alpha+\frac{1}{2}}1_{\lambda t+\sqrt{1+t}\leq x\leq\mu t-\sqrt{1+t}},\label{45j2}
\ema
where
$$
\mathrm{T}^{k}(t)=\int^{t}_0\frac{1}{(1+s)^k}ds=
\left\{\bal
1-\frac{1}{(1+t)^{k-1}}, & k>1,\\
\mathrm{ln}(1+t), & k=1,\\
\frac{1}{(1+t)^{k-1}}-1, & k<1.
\ea\right.
$$
\end{lem}


\begin{lem}[\cite{LI-6}]\label{mvpbpw10}
For any given $\beta,\gamma\ge 0$ and $D>0$, there exists a constant $C>0$ such that
\bma
J^{\beta,\gamma}(t,x;0,t;\lambda,D)\leq C(1+t)^{-\beta}B_{\gamma}(t,x-\lambda t).\label{CJ1}
\ema

In particular, for any $\alpha\geq0$,
\bq
J^{\beta,\gamma}(t,x;0,t;\lambda,D)\leq CI^{\alpha,\beta,\gamma}\(t,x;0,t;\lambda,\mu,C_0\).\label{CJ1j1}
\eq
\end{lem}

\begin{lem}\label{mvpbpw11}
For any given $\beta\geq0$, there exists  constant $C>0$ such that
\bq\label{C4}
\|S^tg_0(x)\|_{L^\infty_{v,\beta}}\leq Ce^{-\frac{2\nu_0t}{3}}\max_{y\in\mathbb{R}}e^{-\frac{\nu_0|x-y|}{3}}\|g_0(y)\|_{L^\infty_{v,\beta}},
\eq
where $S^t$ and $\nu_0$ are defined by \eqref{St} and \eqref{nuv}. In particular, if $g_0(x,v)$ satisfies
$$
\|g_0(x)\|_{L^\infty_{v,\beta}}\leq C(1+|x|^2)^{-\gamma},\quad \gamma>0,
$$
then
\bq\label{Cb41}
\|W_{k}(t)\ast g_0(x)\|_{L^\infty_{v,\beta}}\leq Ce^{-\frac{\nu_0t}{2}}(1+|x|^2)^{-\gamma},
\eq
where $W_k(t,x),\ k\geq0$ is defined by \eqref{28}.
\end{lem}

\begin{proof}
\eqref{C4} was proved in Lemma 4.2 of \cite{LI-5}.
From \eqref{28}, we have
$$
W_k(t)\ast g_0(x)=\sum^{3k}_{i=0}J_{i}(t,x),
$$
where
\bq
\begin{split}
&J_0(t,x)=S^tg_0(x,v)=e^{-\nu(v)t}g_0(x-vt,v),\\
&J_k(t,x)=\int^t_0S^{t-s}(K+v_1\dx (I-\partial_{xx})^{-1}P^1_0)J_{k-1}ds,\quad k\geq1.
\end{split}
\eq

By \eqref{C4}, we have
\bq
\|J_0(t,x)\|_{L^{\infty}_{v,\beta}}\leq Ce^{-\frac{2\nu_0t}{3}}(1+|x|^2)^{-\gamma}.\label{C4j1}
\eq

Noting that the Green's function of $(I-\partial_{xx})^{-1}$ is $\frac{1}{2}e^{-|x|}$, we obtain
\bma
&\quad |\dx (I-\partial_{xx})^{-1}(J_0(t,x),\chi_0)| \nnm\\
&\leq Ce^{-\frac{2\nu_0t}{3}}\int_{\mathbb{R}}e^{-|y|}(1+|x-y|^2)^{-\gamma}dy \nnm\\
&\leq Ce^{-\frac{2\nu_0t}{3}}(1+|x|^2)^{-\gamma}. \label{C4j2}
\ema

Moreover, it follows from \eqref{C4}, \eqref{C4j1} and \eqref{C4j2} that
\bmas
\|J_1(t,x)\|_{L^{\infty}_{v,\beta}}&=\bigg\|\int_{0}^{t}S^{t-s}(KJ_0+v_1\dx (I-\partial_{xx})^{-1}P^1_0J_0)ds\bigg\|_{L^{\infty}_{v,\beta}}\\
&\leq C(1+|x|^2)^{-\gamma}\int_0^te^{-\frac{2\nu_0(t-s)}{3}}e^{-\frac{2\nu_0s}{3}} ds\\
&\leq Cte^{-\frac{2\nu_0t}{3}}(1+|x|^2)^{-\gamma}.
\emas

By a similar argument as  above, we can obtain
$$
\|J_k(t,x)\|_{L^{\infty}_{v,\beta}}\leq Ct^ke^{-\frac{2\nu_0t}{3}}(1+|x|^2)^{-\gamma},\quad \forall k\ge 1,
$$
which proves \eqref{Cb41}. The proof of the lemma is completed.
\end{proof}


\begin{lem}\label{mvpbpw12}
For given $\alpha,\gamma\geq0$ and $\lambda\in\R$, if the function $F(t,x,v)$ satisfies
$$
\|F(t,x,v)\|_{L^{\infty}_{v,\beta-1}}\leq C(1+t)^{-\alpha}B_{\gamma}(t,x-\lambda t),
$$
then we have
\bma
&\bigg\|\int^t_0S^{t-s}F(s,x)ds\bigg\|_{L^{\infty}_{v,\beta}}\leq C(1+t)^{-\alpha}B_{\gamma}(t,x-\lambda t),\label{CbS1}\\
&\bigg\|\int^t_0W_{k}(t-s)\ast F(s,x)ds\bigg\|_{L^{\infty}_{v,\beta}}\leq C(1+t)^{-\alpha}B_{\gamma}(t,x-\lambda t),\label{CbS2}
\ema
where $W_k(t,x),\ k\geq0$ is defined by \eqref{28}.
\end{lem}
\begin{proof}
By Lemma \ref{mvpbpw11}, we have
\bmas
&\quad\nu(v)^{\beta}\int^t_0|S^{t-s}F(s,x,v)|ds\\
&\leq C\int^t_0e^{-\frac{2\nu(v)(t-s)}{3}}\nu(v)\max_{y\in\mathbb{R}}e^{-\frac{\nu_0|x-y|}{3}}\|F(s,y)\|_{L^{\infty}_{v,\beta-1}}ds\\
&\leq C\int^t_0e^{-\frac{2\nu(v)(t-s)}{3}}\nu(v)(1+s)^{-\alpha}\max_{y\in\mathbb{R}}e^{-\frac{\nu_0|x-y|}{3}}B_{\gamma}(s,y-\lambda s)ds.
\emas

Since
\bmas
 e^{-\frac{\nu_0(t-s)}3}e^{-\frac{\nu_0|x-y|}{3}}B_{\gamma}(s,y-\lambda s)&\le e^{-\frac{|x-y-\lambda(t- s)|}{3D}}B_{\gamma}(s,y-\lambda s)\\
 &\le e^{-\frac{|x- \lambda t|}{6D}}+CB_{\gamma}(t,x-\lambda t),
 \emas
with $D=\nu_0^{-1}\max\{1,|\lambda|\}$, we have
\bmas
&\quad\nu(v)^{\beta}\intt | S^{t-s}F_1(s,x,v)|ds\nnm\\
&\le C\intt e^{-\frac{ \nu(v)(t-s)}3}\nu(v)(1+s)^{-\frac{\alpha}2}ds \(e^{-\frac{|x-\lambda t|}{6D}}+B_{\gamma}(t,x-\lambda t)\) \\
&\le C(1+t)^{-\frac{\alpha}2} \(e^{-\frac{|x-\lambda t|}{6D}}+B_{\gamma}(t,x-\lambda t)\),
\emas
which implies \eqref{CbS1}. By a similar argument as  above, we can obtain \eqref{CbS2}.
\end{proof}




With the help of Theorem \ref{mvpbth1} and Lemma \ref{mvpbpw5}--\ref{mvpbpw12}, we can perform the following proof for Theorem \ref{mvpbth2}.

\medskip

\noindent\emph{Proof of Theorem 1.2.} Let $f$ be a solution to the IVP problem \eqref{mVPB8}--\eqref{mVPB10} for $t>0$. According to the Duhamel's principle, the solution can be expressed as
\bq\label{H0}
f(t,x)=G(t)\ast f_0+\int^t_0G(t-s)\ast H(s)ds+\int^t_0G(t-s)\ast \Gamma(f,f)ds,
\eq
where $H$ is the nonlinear term containing the electric potential $\Phi(t,x)$  defined by
\be\label{H01}
H =\frac{1}{2}(v_1\dx \Phi)f-\dx \Phi\partial_{v_1}f+v_1\sqrt{M}\dx (I-\partial_{xx})^{-1}(e^{-\Phi}+\Phi-1).
\ee

Define
\bq
\begin{split}
Q(t)=\sup_{x\in\mathbb{R},0\leq s\leq t}\bigg\{
&(\|f\|_{L^{\infty}_{v,3}}+\|\partial_{v_1}f\|_{L^{\infty}_{v,2}}+|\Phi|)\bigg((1+s)^{-\frac{1}{2}}\sum^{1}_{i=-1}B_{\frac{1}{2}}(s,x-\beta_is)\bigg)^{-1}\\
&+(\|\dx f\|_{L^{\infty}_{v,3}})\bigg((1+s)^{-1}\sum^{1}_{i=-1}B_{\frac{1}{2}}(s,x-\beta_is)\bigg)^{-1}\\
&+(|\dx \Phi|+|\partial_t\Phi|)\bigg((1+s)^{-1}\sum^{1}_{i=-1}B_{\frac{1}{2}}(s,x-\beta_is)\bigg)^{-1}\\
& +(1+s)^{\frac{3}{4}}\sqrt{H_{4,3}(f)}(s)\bigg\}.\nonumber
\end{split}
\eq

From \cite{LI-5}, it holds that for any $\beta\in \N^3$ and $\gamma\ge 0$,
\bq\label{gamma1}
\|\partial_v^{\beta}\Gamma(f,g)\|_{L^{\infty}_{v,\gamma-1}}\leq C\sum_{|\beta_1|+|\beta_2|\leq|\beta|}\|\partial_v^{\beta_1}f\|_{L^{\infty}_{v,\gamma}}\|\partial_v^{\beta_2}g\|_{L^{\infty}_{v,\gamma}}.
\eq

By \eqref{gamma1}, we obtain that  for $0\le s\le t$ and $\alpha=0,1,$
\bma
\|\dxa \Gamma(f,f)(s,x)\|_{L^{\infty}_{v,2}}&\leq C\sum_{\alpha'\leq\alpha}\|\dx^{\alpha'}f(s,x)\|_{L^{\infty}_{v,3}}\|\dx^{\alpha-\alpha'}f(s,x)\|_{L^{\infty}_{v,3}} \nnm\\
&\leq CQ^2(t)(1+s)^{-\frac{2+\alpha}{2}}\sum^{1}_{i=-1}B_{1}(s,x-\beta_is) . \label{H2}
\ema

Moreover, noting that for $0\le s\le t$ and $\alpha=0,1,$
\bma
 &|\dxa\dx (I-\partial_{xx})^{-1}(e^{-\Phi}+\Phi-1)|=|\dxa (I-\partial_{xx})^{-1}(\dx \Phi(e^{-\Phi}-1))|\nnm\\
 &\leq C|\dxa e^{-|x|}|\ast(e^{\|\Phi\|_{L^{\infty}_x}}|\dx \Phi| |\Phi|)\leq CQ^2(t)(1+s)^{-\frac{3}{2}}\sum^{1}_{i=-1}B_{1}(s,x-\beta_is),\label{ph2}
\ema
we have
\bma
\|\dxa H(s,x)\|_{L^{\infty}_{v,2}}&\leq C\sum_{\alpha'\le \alpha}(|\dx^{\alpha'}\dx \Phi| \|\dx^{\alpha-\alpha'}f\|_{L^{\infty}_{v,3}}+|\dx^{\alpha'}\dx \Phi| \|\dx^{\alpha-\alpha'}\partial_{v_1}f\|_{L^{\infty}_{v,2}})\nnm\\
&\quad+C|\dxa e^{-|x|}|\ast(e^{\|\Phi\|_{L^{\infty}_x}}|\dx \Phi| |\Phi|)\nnm\\
&\leq CQ^2(t)(1+s)^{-\frac{3}{2}-\frac{5\alpha}{12}}\sum^{1}_{i=-1}B_{1-\frac{\alpha}{6}}(s,x-\beta_is), \label{MH1}
\ema
where we have used (Gagliardo-Nirenberg interpolation inequality)
\bma
 \|\dx \partial_{v_1}f(s,x)\|_{L^{\infty}_{v,2}}&\leq C\|\dx \Tdv^3(w^2f)\|^{\frac{1}{3}}_{L^2_v}\|\dx (w^2f)\|^{\frac{2}{3}}_{L^{\infty}_v}+C\|\dx (wf)\|_{L^{\infty}_v}\nnm\\
& \leq CQ(t)(1+s)^{-\frac{11}{12}}\sum^3_{i=-1}B_{\frac{1}{3}}(s,x-\beta_is), \\
 |\partial_{xx}\Phi(s,x)|&=|\dx (I-\partial_{xx})^{-1}\dx n+\dx (I-\partial_{xx})^{-1}(\dx \Phi(e^{-\Phi}-1))|\nnm\\
& \leq |\dx e^{-|x|}|\ast|\dx n|+|\dx e^{-|x|}|\ast(e^{\|\Phi\|_{L_x^{\infty}}}|\dx \Phi| |\Phi|)\nnm\\
& \leq CQ(t)(1+s)^{-1}\sum^{1}_{i=-1}B_{\frac{1}{2}}(s,x-\beta_is).
\ema

First, we show the pointwise estimate of $\dxa f$. By \eqref{H0}, we have
\bma
\dxa f&=\dxa G(t)\ast f_0+\int^t_0\dxa G(t-s)\ast H(s)ds+\int^t_0\dxa G(t-s)\ast \Gamma(f,f)ds\nnm\\
&=:I_1+I_2+I_3, \quad \alpha=0,1.
\ema

We estimate $I_j$, $j=1,2,3$ as follows. By Theorem \ref{mvpbth1},  $I_1$ can be decomposed into
\bma
I_1&=\dxa G_1(t)\ast f_0+G_2(t)\ast \dxa f_0+W_2(t)\ast \dxa f_0\nnm\\
&=:I_1^1+I_1^2+I_1^3\nnm.
\ema

For $I_1^1$, it follows from Theorem \ref{mvpbth1} and Lemma \ref{mvpbpw8} that
\bma
\|I_1^1\|&\leq C\delta_0\sum^{1}_{i=-1}\int_{\R}(1+t)^{-\frac{1+\alpha}{2}}e^{-\frac{|x-y-\beta_it|^2}{D(1+t)}}(1+|y|^2)^{-\gamma_0}dy\nnm\\
&\quad+C\delta_0\int_{\R}e^{-\frac{|x-y|+t}{D}}(1+|y|^2)^{-\gamma_0}dy\nnm\\
&\leq C\delta_0(1+t)^{-\frac{1+\alpha}{2}}\sum^{1}_{i=-1}B_{\frac{1}{2}}(t,x-\beta_it).\label{FI1}
\ema

For $I_1^2$ and $I_1^3$, we can obtain by Theorem \ref{mvpbth1} and Lemma \ref{mvpbpw11} that
\be
 \|I_1^2\|+\|I_1^3\|\leq C\(e^{- \frac{t}{D}}+e^{-\frac{\nu_0t}{2}}\)(1+|x|^2)^{-\gamma_0}.\label{FI2}
\ee

Thus, it follows from \eqref{FI1}--\eqref{FI2} that
\be
\|I_1\|\leq C\delta_0(1+t)^{-\frac{1+\alpha}{2}}\sum^{1}_{i=-1}B_{\frac{1}{2}}(t,x-\beta_it).\label{nonlin1}
\ee

By Theorem \ref{mvpbth1}, we decompose $I_3$ into
\bma
I_3&=\int^t_0\dxa G_1(t-s)\ast \Gamma(f,f)ds+\int^t_0G_2(t-s)\ast \dxa \Gamma(f,f)ds\nnm\\
&\quad+\int^t_0W_2(t-s)\ast \dxa \Gamma(f,f)ds\nnm\\
&=:I_3^1+I_3^2+I_3^3.\nnm
\ema

For $I_3^1$,  noting that $P_0\Gamma(f,f)=0$, we can obtain by Theorem \ref{mvpbth1}, Lemmas \ref{mvpbpw9}--\ref{mvpbpw10} and \eqref{H2} that
\bma
\|I_3^1\|&\leq\bigg\|\int^{\frac{t}{2}}_0\dxa G_1(t-s)\ast \Gamma(f,f)ds\bigg\|+\bigg\|\int_{\frac{t}{2}}^tG_1(t-s)\ast \dxa \Gamma(f,f)ds\bigg\|\nnm\\
&\leq CQ^2(t)\sum^{1}_{i,j=-1}I^{\frac{2+\alpha}{2},1,1}\(t,x;0,\frac{t}{2};\beta_i,\beta_j,D\)\nnm\\
&\quad+CQ^2(t)\sum^{1}_{i,j=-1}I^{1,\frac{2+\alpha}{2},1}\(t,x;\frac{t}{2},t;\beta_i,\beta_j,D\)\nnm\\
&\leq CQ^2(t)(1+t)^{-\frac{1+\alpha}{2}}\sum^{1}_{i=-1}B_{\frac{1}{2}}(t,x-\beta_it). \label{I13}
\ema

For $I_3^2$ and $I_3^3$, we have by \eqref{H2}, Lemma \ref{mvpbpw10} and Lemma \ref{mvpbpw12} that
\be
 \|I_3^2\|+\|I_3^3\|\leq CQ^2(t)(1+t)^{-\frac{1+\alpha}{2}}\sum^{1}_{i=-1}B_{\frac{1}{2}}(t,x-\beta_it).\label{H4j}
\ee

Thus, it follows from \eqref{I13}--\eqref{H4j} that
\be
\|I_3\|\leq CQ^2(t)(1+t)^{-\frac{1+\alpha}{2}}\sum^{1}_{i=-1}B_{\frac{1}{2}}(t,x-\beta_it).\label{nonlin3}
\ee

To estimate $I_2$, decompose
\bma
I_2&=\int^t_0\dxa G_1(t-s)\ast P_0H(s)ds+\int^t_0\dxa G_1(t-s)\ast P_1H(s)ds\nnm\\
&\quad+\int^t_0G_2(t-s)\ast \dxa H(s)ds+\int^t_0W_2(t-s)\ast \dxa H(s)ds\nnm\\
&=:I_2^1+I_2^2+I_2^3+I_2^4.\nnm
\ema

For $I^1_2$, we have
\be
I_2^1=\int^{t/2}_0\dxa G_1(t-s)\ast P_0Hds+ \int^t_{t/2}\dxa G_1(t-s)\ast P_0Hds=:I_2^{11}+I_2^{12}. \label{FI40}
\ee

For $I_2^{12}$, we obtain by Theorem \ref{mvpbth1} and \eqref{MH1} that
\bma
\|I_2^{12}\|&\le CQ^2(t)\sum^{1}_{i,j=-1} I^{\frac{1+\alpha}{2},\frac{3}{2},1}\(t,x;\frac{t}{2},t;\beta_i,\beta_j,D\) \nnm\\
&\leq CQ^2(t)(1+t)^{-\frac{1+\alpha}{2}}\sum^{1}_{i=-1}B_{\frac{1}{2}}(t,x-\beta_it).
\ema

To estimate $I_2^{11}$,  write
$$
P_0H = n\dx \Phi\chi_1+\sqrt{\frac23} m_1\dx \Phi\chi_4+\chi_1\dx (I-\partial_{xx})^{-1}(e^{-\Phi}+\Phi-1),
$$
where $n=(f,\chi_0)$ and $m_1=(f,\chi_1).$ 
Since
\be
n=e^{-\Phi}-1+\partial_{xx}\Phi,\quad \dx m_1=-\dt n, \label{aaa}
\ee
it follows that
\bmas
 n\dx \Phi&=\frac{1}{2}\dx (\dx \Phi)^2-\dx (e^{-\Phi}+\Phi-1),\\
 m_1\dx \Phi&=\dx (m_1 \Phi)+\Phi \partial_{xx}\partial_t\Phi-\Phi\partial_t\Phi e^{-\Phi} \\
&=\dx (m_1 \Phi )+\partial_t\Big( \Phi \partial_{xx} \Phi-\frac{1}{2}\Phi^2 \Big)-\partial_t\Phi\partial_{xx} \Phi - \Phi\partial_t\Phi (e^{-\Phi}-1).
\emas

Thus, we rewrite $P_0H$ as
\be
P_0H=H_1+\dx H_2+\dt H_3 , \label{P0H}
\ee
where
\bmas
H_1&=-\sqrt{\frac23}\[\partial_t\Phi\partial_{xx} \Phi - \Phi\partial_t\Phi (e^{-\Phi}-1)\]\chi_4,\\
H_2&=\frac{1}{2} (\dx \Phi)^2\chi_1- (e^{-\Phi}+\Phi-1)\chi_1+\sqrt{\frac23}(m_1 \Phi )\chi_4,\\
H_3&=\sqrt{\frac23}\Big(\Phi \partial_{xx} \Phi-\frac{1}{2}\Phi^2\Big)\chi_4.
\emas

It follows that
\bma
I_2^{11}&=\int^{\frac{t}{2}}_0\partial^{\alpha}_{x}G_1(t-s)\ast H_1ds+ \int^{\frac{t}{2}}_0 \partial^{\alpha+1}_{x}G_1(t-s)\ast H_2ds \nnm\\
&\quad+\int^{\frac{t}{2}}_0\dt\partial^{\alpha}_{x}G_1(t-s)\ast H_3ds+  \partial^{\alpha}_{x}G_1(t/2 )\ast H_3(0) \nnm\\
&=:I_2^{111}+I_2^{112}+I_2^{113}+I_2^{114},
\ema
where we have used that $ G_{1}(0,x)=0$ because $G_{1}=G_{L,0}1_{\{|x|\le 6t\}}$ is supported in $|x|\le 6t$.

For $I_2^{111}$ and $I_2^{112}$, we have by Theorem \ref{mvpbth1} and Lemmas \ref{mvpbpw9}--\ref{mvpbpw10} that
\bma
\|I_2^{111}\|&\leq CQ^2(t)\sum^{1}_{i,j=-1}I^{\frac{1+\alpha}{2},2,1}\(t,x;0,\frac{t}{2};\beta_i,\beta_j,D\)\nnm\\
&\leq CQ^2(t)(1+t)^{-\frac{1+\alpha}{2}}\sum^{1}_{i=-1}B_{\frac{1}{2}}(t,x-\beta_it),\label{I425}
\\
\|I_2^{112}\|&\leq CQ^2(t)\sum^{1}_{i,j=-1}I^{\frac{2+\alpha}{2},1,1}\(t,x;0,\frac{t}{2};\beta_i,\beta_j,D\)\nnm\\
&\leq CQ^2(t)(1+t)^{-\frac{1+\alpha}{2}}\sum^{1}_{i=-1}B_{\frac{1}{2}}(t,x-\beta_it).\label{I421}
\ema

By \eqref{GL0} and \eqref{G1xt}, we can obtain
\bmas
&\quad\partial_tG_1(t,x)\nnm\\
&=\sum^{3}_{j=-1}\frac{1}{\sqrt{2\pi}}\int_{|\eta|<r_{0}}\lambda_j(\eta)e^{ix\eta +\lambda_j(\eta)t}\psi_j(\eta)\otimes\bigg\langle \psi_j(\eta)+\frac{1}{1+\eta^2}P^1_0\psi_j(\eta)\bigg|d \eta.
\emas

Repeating the similar arguments as Lemma \ref{mvpbgf2}, we have for $\alpha\ge 0$ that
\be
\|\partial_t\dxa G_1(t,x)\|\leq C(1+t)^{-\frac{2+\alpha}{2}}\sum^{1}_{i=-1}e^{-\frac{|x-\beta_it|^2}{D(1+t)}}+Ce^{-\frac{|x|+t}{D}}. \label{dtg}
\ee

Then, it follows from \eqref{dtg}, Theorem \ref{mvpbth1} and Lemmas \ref{mvpbpw8}--\ref{mvpbpw10} that
\bma
\|I_2^{113}\|+\|I_2^{114}\|&\leq C\delta^2_0 \sum^{1}_{i=-1}(1+t)^{-\frac{1+\alpha}2}B_{2\gamma_0}(t,x-\beta_i t) \nnm\\
&\quad+CQ^2(t)\sum^{1}_{i,j=-1} I^{\frac{2+\alpha}{2},1,1}\(t,x;0,\frac{t}{2};\beta_i,\beta_j,D\) \nnm\\
&\leq C(\delta^2_0+Q^2(t))(1+t)^{-\frac{1+\alpha}{2}}\sum^{1}_{i=-1}B_{\frac{1}{2}}(t,x-\beta_it),\label{I432}
\ema
where we have used
$$
| \Phi(0,x)|,|\pt_{xx} \Phi(0,x)|
 \leq C\delta_0(1+|x|)^{-\gamma_0} .
$$

By taking summation $I_2^{12}+\sum^4_{i=1}I_2^{11i}$, we have
\be\label{I42}
\|I_2^{1}\|\leq C(\delta_0+Q^2(t))(1+t)^{-\frac{1+\alpha}{2}}\sum^{1}_{i=-1}B_{\frac{1}{2}}(t,x-\beta_it).
\ee

For $I_2^2$, we have by  Theorem \ref{mvpbth1} and Lemma \ref{mvpbpw10}  that
\bma
\|I_2^2\|&\leq CQ^2(t)\sum^{1}_{i,j=-1}I^{\frac{2+\alpha}{2},\frac{3}{2},1}\(t,x;0,t;\beta_i,\beta_j,D\)\nnm\\
&\leq CQ^2(t)(1+t)^{-\frac{1+\alpha}{2}}\sum^{1}_{i=-1}B_{\frac{1}{2}}(t,x-\beta_it).
\ema

For $I_2^3$ and $I_2^4$, we have by \eqref{MH1}, Theorem \ref{mvpbth1}, Lemma \ref{mvpbpw10} and Lemma \ref{mvpbpw12} that
\be
\|I_2^3\|+\|I_2^4\|\leq CQ^2(t)(1+t)^{-\frac{1+\alpha}{2}}\sum^{1}_{i=-1}B_{\frac{1}{2}}(t,x-\beta_it).\label{FI7}
\ee

Thus,
\be
\|I_2\|\leq CQ^2(t)(1+t)^{-\frac{1+\alpha}{2}}\sum^{1}_{i=-1}B_{\frac{1}{2}}(t,x-\beta_it).\label{nonlin2}
\ee

By taking summation $\eqref{nonlin1}+\eqref{nonlin3}+\eqref{nonlin2}$, we obtain
\bq\label{PDF}
\|\dxa f\|\leq C(\delta_0+Q^2(t))(1+t)^{-\frac{1+\alpha}{2}}\sum^{1}_{i=-1}B_{\frac{1}{2}}(t,x-\beta_it).
\eq


By \eqref{mVPB9}, we have
$$
\Phi=-(I-\partial_{xx})^{-1}n+(I-\partial_{xx})^{-1}(e^{-\Phi}+\Phi-1),\\
$$
which implies that
\bma
 |\Phi(t,x)|&\leq Ce^{-|x|}\ast |n|+Ce^{-|x|}\ast\[e^{\|\Phi\|_{L^{\infty}_x}}|\Phi|^2\]\nnm\\
 &\leq C(\delta_0+Q^2(t))(1+t)^{-\frac{1}{2}}\sum^{1}_{i=-1}B_{\frac{1}{2}}(t,x-\beta_it),\label{p0p}\\
 |\dx \Phi(t,x)|&\leq Ce^{-|x|}\ast |\dx n|+Ce^{-|x|}\ast\[e^{\|\Phi\|_{L^{\infty}_x}}(|\dx \Phi| |\Phi|)\]\nnm\\
 &\leq C(\delta_0+Q^2(t))(1+t)^{-1}\sum^{1}_{i=-1}B_{\frac{1}{2}}(t,x-\beta_it),\label{pxp}\\
 |\partial_t\Phi(t,x)|&\leq Ce^{-|x|}\ast|\dx m_1|+Ce^{-|x|}\ast\[e^{\|\Phi\|_{L^{\infty}_x}}(|\partial_t\Phi| |\Phi|)\]\nnm\\
 &\leq C(\delta_0+Q^2(t))(1+t)^{-1}\sum^{1}_{i=-1}B_{\frac{1}{2}}(t,x-\beta_it).\label{ptp}
\ema

By \eqref{mVPB8}, we have
\bq\label{H9}
\partial_tf+v_1\dx f+\nu(v)f=Kf+v_1\dx \Phi \chi_0+H+\Gamma(f,f).
\eq

Then, we can represent $\dx^\alpha f$ as
\bq\label{H10}
\dx^\alpha f =S^t\dx^\alpha f_0+\int^t_0S^{t-s}\dx^\alpha(Kf+v_1\dx \Phi \chi_0+H+\Gamma(f,f))ds.
\eq

By Lemma \ref{mvpbpw11}, it holds that
\bq\label{H11}
\|S^t\dx^\alpha f_0(x)\|_{L^{\infty}_{v,3}}\leq C\delta_0e^{-\frac{2\nu_0t}{3}}(1+|x|^2)^{-\gamma_0}.
\eq

By \eqref{PDF}, we obtain
$$
\|\dx^\alpha Kf \|_{L^{\infty}_{v,0}}\leq C\|\dx^\alpha f\|\leq C(\delta_0+Q^2(t))(1+t)^{-\frac{1+\alpha}{2}}\sum^{1}_{i=-1}B_{\frac{1}{2}}(t,x-\beta_it),
$$
which, together with Lemma \ref{mvpbpw11} and \eqref{MH1}, implies that
\bma\label{H12}
&\bigg\|\int^t_0S^{t-s}\dx^\alpha(Kf+v_1\dx \Phi \chi_0+H+\Gamma(f,f))ds\bigg\|_{L^{\infty}_{v,1}}\nnm\\
&\leq C(\delta_0+Q^2(t))(1+t)^{-\frac{1+\alpha}{2}}\sum^{1}_{i=-1}B_{\frac{1}{2}}(t,x-\beta_it).
\ema

It follows from \eqref{H10}--\eqref{H12} that
\bq
\|\dx^\alpha f(t,x)\|_{L^{\infty}_{v,1}}\leq C(\delta_0+Q^2(t))(1+t)^{-\frac{1+\alpha}{2}}\sum^{1}_{i=-1}B_{\frac{1}{2}}(t,x-\beta_it).
\eq

By induction and
$$
\|\dx^\alpha Kf(t,x)\|_{L^{\infty}_{v,k}}\leq C\|\dx^\alpha f(t,x)\|_{L^{\infty}_{v,k-1}},\quad k\geq1,
$$
we have
\bq\label{H13}
\|\dxa f(t,x)\|_{L^{\infty}_{v,3}}\leq C(\delta_0+Q^2(t))(1+t)^{-\frac{1+\alpha}{2}}\sum^{1}_{i=-1}B_{\frac{1}{2}}(t,x-\beta_it).
\eq

Taking the derivative $\partial_{v_1}$ to \eqref{H9}, we have
\bq
\partial_t\partial_{v_1}f+v_1\dx \partial_{v_1}f+\nu(v)\partial_{v_1}f= H_4+\partial_{v_1}\Gamma(f,f),
\eq
where
\bq
\begin{split}
 H_4&=-\dx f-\partial_{v_1}\nu(v)f+\partial_{v_1}(v_1\sqrt{M})\dx \Phi+\partial_{v_1}(Kf)\\
&\quad+\frac{1}{2}\dx \Phi (f+v_1\partial_{v_1}f)-\dx \Phi (\partial^2_{v_1}f).\nonumber
\end{split}
\eq

Thus, we can represent $\partial_{v_1}f$ as
\bq\label{H14}
\partial_{v_1}f(t,x)=S^t\partial_{v_1}f_0+\int^t_0S^{t-s}( H_4+\partial_{v_1}\Gamma(f,f))ds.
\eq

It follows from \eqref{H13} that
\bma\label{H141}
&\quad\| H_4(s,x)\|_{L^{\infty}_{v,1}}+\|\partial_{v_1}\Gamma(f,f)(s,x)\|_{L^{\infty}_{v,1}} \nnm\\
&\leq C\[(\delta_0+Q^2(t))(1+s)^{-\frac{1}{2}}+Q^2(t)(1+s)^{-\frac{5}{4}}\]\sum^{1}_{i=-1}B_{\frac{1}{2}}(s,x-\beta_is),
\ema
where we have used
\bmas
& \|(\partial_{v_1}K) f\|_{L^{\infty}_{v,k}}\leq C\| f\|_{L^{\infty}_{v,k}} ,\nonumber\\
&\|\partial^2_{v_1}f\|_{L^{\infty}_{v,1}}\leq C\|w\partial^{2}_{v_1}f\|_{H^2_{v}}\leq C(1+t)^{-\frac{3}{4}}Q(t).
\emas

By \eqref{CbS1} and \eqref{H14}--\eqref{H141}, we obtain
\bq\label{H15}
\|\partial_{v_1}f\|_{L^{\infty}_{v,2}}\leq C(\delta_0+Q^2(t))(1+t)^{-\frac{1}{2}}\sum^{1}_{i=-1}B_{\frac{1}{2}}(t,x-\beta_it).
\eq

We claim that $Q(t)\leq C\delta_0$, which implies that $E_4(f)\leq C\delta_0^2(1+t)^{-\frac12}$. By \eqref{ENI2} and $d_1\mathcal{H}_{4,3}(f)\leq D_{4,3}(f)$ with $d_1>0$, we have
\bma
\mathcal{H}_{4,3}(f)(t)&\leq Ce^{-d_1 t}\mathcal{H}_{4,3}(f_0)+C \int^t_0e^{-d_1(t-s)}\|\dx P_0f(s)\|^2_{L^2_{x,v} }ds\nnm\\
&\quad+C \delta_0^2\int^t_0e^{-d_1(t-s)}(1+s)^{-\frac{1}{2}}\|P_0f(s)\|^2_{L^2_{v}(L^{\infty}_x)}ds\nnm\\
&\leq C\delta_0^2e^{-d_1 t}+C (\delta_0+Q^2(t))^2 \int^t_0e^{-d_1(t-s)}(1+s)^{-\frac{3}{2}}ds\nnm\\
&\leq C(1+t)^{-\frac{3}{2}}(\delta_0+Q^2(t))^2.\label{H16}
\ema

Combining  \eqref{p0p}--\eqref{ptp}, \eqref{H13} and \eqref{H15}--\eqref{H16}, we have
$$
Q(t)\leq C\delta_0+CQ^2(t),
$$
from which \eqref{NL5}--\eqref{NL4} can be verified so long as $\delta_0>0$ is small enough.

\medskip
\noindent {\bf Acknowledgements:}  This work is supported by the National Science Fund for Excellent Young Scholars No. 11922107, the National Natural Science Foundation of China  grants No. 12171104,  Guangxi Natural Science Foundation No. 2019JJG110010, and the special foundation for Guangxi Ba Gui Scholars.


\end{document}